\documentclass[12pt]{article}
\usepackage{amssymb, amsfonts, amsmath}
\usepackage{mathrsfs}
\usepackage{graphicx}
\usepackage[usenames,dvipsnames]{xcolor}
\usepackage{tikz}
\usepackage{float}
\usepackage{enumitem}
\usepackage[ruled, titlenotnumbered]{algorithm2e}
\usepackage{comment}

\usepackage{array}
\newcolumntype{C}[1]{>{\centering\let\newline\\\arraybackslash\hspace{0pt}}m{#1}}

\providecommand{\U}[1]{\protect\rule{.1in}{.1in}}

\providecommand{\gb}{\gamma_b}
\providecommand{\mlp}{\operatorname{mp}}

\providecommand{\diam}{\operatorname{diam}}
\providecommand{\rad}{\operatorname{rad}}

\renewcommand{\mod}{\operatorname{mod}\ }
\renewcommand{\vec}[1]{\mathbf{#1}}

\setlist{nolistsep}

\newtheorem{theorem}{Theorem}

\newtheorem{algo}{Algorithm}

\newtheorem{coro}[theorem]{Corollary}

\newtheorem{lemma}[theorem]{Lemma}

\newtheorem{problem}{Problem}
\newtheorem{prop}[theorem]{Proposition}

\newenvironment{proof}[1][Proof]{\noindent\textbf{#1.} }{\ \hfill \rule{0.5em}{0.5em}}
    \newcounter{example}
    \newenvironment{example}[1][]{\refstepcounter{example}\par\medskip\noindent%
       \textbf{Example~\theexample. #1} \rmfamily}{\medskip}

\usetikzlibrary{fit,positioning}
\tikzstyle{std}=[ circle, draw=black,fill=black,thick, inner sep=1pt, minimum size=1.5mm]
\tikzstyle{ir}=[ circle, draw=black,fill=green,thick,  inner sep=2pt, minimum size=2mm]
\tikzstyle{mp}=[circle, draw=black,fill=Dandelion,thick,  inner sep=2pt, minimum size=2mm]
\tikzstyle{bred}=[circle, draw=black,fill=red,thick,  inner sep=2pt, minimum size=2mm]
\tikzstyle{smred}=[ circle, draw=black,fill=red,thick,  inner sep=1pt, minimum size=1.5mm]
\tikzstyle{bblue}=[ circle, draw=black,fill=blue,thick,  inner sep=1pt, minimum size=1.5mm]

\tikzstyle{tr}=[color=black, style=dotted]
\tikzstyle{sp}=[color=ProcessBlue, line width = 2pt]

\begin{document}

\title{\textbf{Multipackings of Graphs}}
\author{L. E. Teshima\\Department of Mathematics and Statistics\\University of Victoria, P. O. Box 3045, Victoria, BC\\\textsc{Canada} V8W 3P4\\{lteshima@uvic.ca}}
\maketitle

\begin{abstract}
A set $M\subseteq V$ is called a \emph{multipacking} of a graph $G=(V,E)$ if,
for each $v\in V$ and each $s$ such that $1\leq s\leq\operatorname{diam}(G)$,
$v$ is within distance $s$ of at most $s$ vertices in $M$. The
\emph{multipacking number}, denoted $\mlp(G)$, is the maximum cardinality of a
multipacking of $G$.  A \emph{dominating broadcast} of $G$ is a function $f:V\rightarrow
\{0,1,\dots,\operatorname{diam}(G)\}$ such that $f(v)\leq e(v)$ (the
eccentricity of $v$) for all $v\in V$ and such that each vertex is within
distance $f(v)$ from a vertex $v$ with $f(v)>0$. The \emph{cost} of a
broadcast $f$ is $\sigma(f)=\sum_{v\in V}f(v)$, and the \emph{broadcast
number} $\gamma_{b}(G)$ is the minimum cost of a dominating broadcast.  
In this paper, we review a variety of recent results in the study of the dual graph parameters $\mlp$ and $\gb$.  
\end{abstract}

\noindent\textbf{Keywords:\hspace{0.1in}}Broadcast Domination, Multipackings, Fractional Multipackings, Farber's Algorithm, Balanced Matrices

\noindent\textbf{AMS Subject Classification 2010:\hspace{0.1in}}05C69, 05C70, 05C72 05C85, 90C27


\section{Introduction}
\label{sec:intro}

A \emph{broadcast} on a connected graph $G=(V,E)$ is a function $f:V\rightarrow\left\{
0,1,\dots,\operatorname{diam}(G)\right\}  $ such that $f(v)\leq e(v)$,
  for all $v\in V$, where $e(v)$ is the eccentricity of $v$. The set of \emph{broadcast vertices} $V_f^+= \{v: f(v)\geq 1\}$ is the set of vertices that transmit the broadcast.  A vertex $v$ is said to \emph{hear} a broadcast if there exists some broadcast vertex $u$ such that $d(u,v) \leq f(u)$.  
The $k$-neighbourhood of a vertex $v$ is the vertex subset $N_k[v]=\{u \in V: d(v,u) \leq k\}$.  If $u$ is a broadcast vertex, $v$ hears the broadcast from $u$ if and only if $v$ is in the \emph{broadcast neighbourhood} $N_{f}[u] = N_{f(u)}[u]$  of $u$.  
We say that $f$ is a \emph{dominating
broadcast } if each vertex of $G$ hears a broadcast.  The \emph{cost} of a broadcast $f$ is
$\sigma(f)=\sum_{v\in V}f(v)$, and the \emph{broadcast number} of $G$ is
$\gamma_{b}(G)$, is the minimum cost of a dominating broadcast. 
A broadcast is \emph{efficient} if each vertex hears exactly one broadcast.  Conversely, a vertex is said so be \emph{over-dominated} if it hears multiple broadcasts.  
A dominating broadcast $f$ of $G$ such that $\sigma
(f)=\gamma_{b}(G)$ is called a $\gamma_{b}$-\emph{broadcast}.

If $f$ is a
dominating broadcast such that $f(v)\in\{0,1\}$ for each $v\in V$, then
$\{v\in V:f(v)=1\}$ is a \emph{dominating set} of $G$; the smallest
cardinality of a dominating set is the \emph{domination number }$\gamma(G)$.  A dominating subset $X\subseteq V(G)$ with $|X| = \gamma(G)$ is called a $\gamma$-\emph{set} of $G$.

Broadcast domination was introduced as a generalization of ordinary domination by Erwin in his 2001 doctoral dissertation as \emph{cost domination} \cite{Ethesis} (see also \cite{Epaper}).   Unlike ordinary domination, the minimum cost dominating broadcast problem can be solved in polynomial time  for general graphs \cite{HL}, and linear time for trees \cite{DabThesis, DDH}.   This has made broadcasting a popular new research topic with many recent publications on broadcasts on trees    \cite{ CHM, DHH,Herke, HM, Lunney, SS}    and general graphs \cite{ BHHM, BZ, BS, BMT, DEHHH,  LThesis}.

In 2013, Brewster, Mynhardt and Teshima \cite{BMT, LThesis} examined the minimum dominating broadcast problem as an integer programming (IP) problem.  The resulting dual property was dubbed a multipacking. Formally, a vertex subset $M$ is a  $k$-\emph{multipacking} if for each $v\in V$ and each integer $s$ for $1 \leq s \leq k$, $|M \cap N_s[v]| \leq s$.  The $k$-\emph{multipacking number} $\mlp_k(G)$ is the maximum cardinality of a $k$-multipacking of $G$.  When $k=\diam(G)$, $M$ is called a \emph{multipacking} and $\mlp_k(G)$ is the \emph{multipacking number}, $\mlp(G)$.

Similarly to broadcasts and domination, multipackings are a generalization of 2-packings. A vertex subset $Y$ of a graph $G=(V,E)$  is a $2$-\emph{packing} if for each $v \in V$, $|Y \cap N[v]| \leq 1$.  Thus, a 2-packing is a 1-multipacking.  The $2$-\emph{packing number} $\rho(G)$ is the size of a maximum 2-packing of $G$. 

Concepts not defined here can be found in \cite{CL, Corn, HHS, HHS2}.

\section{Broadcasts and Multipackings}
\label{sec:broadcast}

\subsection{Integer Programming Formulation}

A dominating broadcast on a graph $G$ can also been viewed as a covering of $%
G$ with $k$-neighbourhoods centred at each broadcast vertex $v$ and each $k\in  \{1,2,\dots, e(v)\}$. Thus a
dominating broadcast can be seen as a collection of balls $\mathcal{B}=\left\{ N_{k}[v]\right\} $
such that for each $u\in V$ there exists some $N_{k}[v]\in \mathcal{B}$ with 
$u\in N_{k}[v]$. 

Suppose $G=(V,E)$ is a graph with $V=\{v_1,v_2,\dots,v_n\}$.  Let $c$, indexed by $(v,k)$ (where $v\in V$ and $1 \leq k \leq e(v)$), be the \emph{cost vector} for the IP, and set each $c_{v,k}=k$.  Furthermore, define the vector $x$, also indexed by $(v,k)$, so that each $x_{v,k}$ is an indicator variable in the IP's solution.   That is, for the optimal broadcast $f$ found by the IP, 

\begin{equation*}
x_{v,k}=\left\{ 
\begin{tabular}{ll}
$1$ & if $f(v)=k$ \\ 
$0$ & otherwise.%
\end{tabular}%
\right.
\end{equation*}

\noindent Finally, let $A$ be the incidence matrix with its $n$ rows indexed by the vertices $v_i$, and its $m$ columns indexed by the pairs $(j,k)$, representing the $k$-neighbourhood of the vertex $v_j$.  The entries of $A$ are therefore,

\begin{align*} a_{i,(j,k)} =
	\begin{cases}
1 & \text{ if } v_i \in N_k[v_{j}], \\
0 & \text{ otherwise.}		
	\end{cases}
\end{align*}

\noindent We call $A$ the \emph{extended neighbourhood matrix} of $G$.  Thus the primal integer program (PIP) for a minimum cost broadcast, and the dual integer program (DIP) for a maximum multipacking, are as below. 

\hspace{5mm}\textbf{The Broadcast PIP: $B-PIP(G):$} 
\begin{align*}
\min & \ c \cdot x \\
\text{s.t.}&\ Ax \geq \vec{1}\\
\ x_{k,v} & \in \{0,1\}.
\end{align*}%

\hspace{5mm}\textbf{The Multipacking DIP: $MP-DIP(G):$} 
\begin{align*}
\max &\ y \cdot \vec{1} \\
\text{s.t.}&\ yA \leq c \\
 \ y_{u} & \in \{0,1\}.
\end{align*}%

\noindent Furthermore by the strong duality theorem of linear programs, we can concluded that for any graph $G$,
\begin{align*}
\mlp(G) \leq \gamma_b(G).
\end{align*}

\subsection{Bounds, Differences and Ratios}

In this section we present some recent results in comparing $\gb$ and $\mlp$.    We begin with one of the first results by Erwin, the bound presented in Proposition \ref{prop:erwinBound}.

\begin{prop}\emph{\cite{Ethesis}}\label{prop:erwinBound}  For every non trivial connected graph $G$,
	\begin{align*} \left\lceil \frac{\diam(G)+1}{3}\right\rceil \leq \gamma_b(G) \leq \min\{\rad(G), \gamma(G)\}. \end{align*}
\end{prop}

\noindent Hartnell and Mynhardt \cite{HartMyn} have expanded this result to include multipackings.

\begin{prop}\emph{\cite{HartMyn}, \cite{LThesis}}\label{prop:mpDiam}
	For any connected graph $G$, 
	\begin{align*} \left\lceil \frac{\diam(G)+1}{3}\right\rceil \leq \mlp(G) \leq \gamma_b(G) \leq \min\{\rad(G), \gamma(G)\}. \end{align*}. 
\end{prop}

\begin{proof}
Let $d = \diam(G)$ and $P = v_0, v_1,\dots, v_d$ be a diametrical path of $G$.  Define $V_i = \{v: d(v,v_0)=i\}$ and $M=\{v_i: i \equiv 0 (\mod 3), i=0,\dots,d \}$.  By our choice of $M$, any  $v_i \in V(P)$ satisfies $|N_s[v_i] \cap M| \leq s$ for all $s \geq 1$.  Consider now any $1 \leq r \leq d$ and any $v\in V_r$.  Since $v_r\in V_r$ is on $P$ and $M \subseteq V(P)$, $N_s[v] \cap M \subseteq N_s[v_r]\cap M$.  Thus $|N_s[v]\cap M| \leq s$ for all $s \geq 1$.  It follows that $\mlp(G) \geq |M| = \left\lceil \frac{\diam(G)+1}{3}\right\rceil$.
\end{proof}

By combining the results of Propositions \ref{prop:erwinBound}  and \ref{prop:mpDiam}, Hartnell and Mynhardt closed an open problem from \cite{BMT} which asked whether the $\gamma_b / \mlp$ ratio could be arbitrary.

\begin{coro}\emph{\cite{HartMyn}}\label{prop:ratio}
For any graph $G$, 
	\begin{align*}
		\gamma_b(G)  /  \mlp(G) < 3.
	\end{align*}
\end{coro}

\begin{proof}
 Since $3 \mlp(G) \geq \diam(G)+1 > \rad(G) \geq \gamma_b(G)$, 
 \begin{align*}
 \gamma_b(G) \leq 3 \mlp(G) - 1,
 \end{align*}
 and so
 \begin{align*}
 	\gamma_b(G)  /  \mlp(G) < 3.
 \end{align*}
\end{proof}

\noindent Hartnell and Mynhardt also offer the following result, which is the only known upper bound for $\gb$ in terms of $\mlp$.

\begin{prop}\emph{\cite{HartMyn}}\label{prop:mpB}
If $G$ is a graph with $\mlp(G) \geq 2$, then $\gamma_b(G) \leq 3 \mlp(G)-2$.   Furthermore, equality is reached for some graphs $G$ with $\mlp(G)=2$.
\end{prop}


Naturally, the study of bounds has journeyed into investigations of equally.  In particular, for which graphs is $\gamma_b=\mlp$?  Trivially, $\gamma_b(P_n) = \mlp(P_n)$, where $P_n$ is the path on $n$ vertices.  A similar result follows for cycles.

\begin{prop}\emph{\cite{LThesis}}\label{prop:mpCycle}
For any cycle $C_n$ with $n \geq 3$, $\mlp(C_n) = \gamma_b(C_n)$ if and only if $n \equiv 0 \ (\mod 3)$.
\end{prop}

\noindent A famous result in domination is the equality of $\gamma(T)$ and $\rho(T)$ for any tree, as shown by Meir and Moon in \cite{MM}.  In 2013, Mynhardt and Teshima extended this result for broadcasts and multipackings.
\begin{theorem}
\label{th:mpTree} \emph{\cite{MT}}  For any tree $T$, $\gamma_{b}(T)=\mlp(T)$.
\end{theorem}
\noindent The original proof for Theorem \ref{th:mpTree} provided in \cite{MT} is quite long and technical; however, it does provide a useful algorithm for finding a maximum multipacking of tree, which we present with example in Appendix \ref{ap:mpTree}.  Instead, we show Brewster and Duchesne's alternative proof using Farber's algorithm in Section \ref{sec:farber}.   

Exploration has also ventured into examination of the $\gamma_b-\mlp$ gap.  To start, we present the following proposition which provides a trivial condition for inequality between $\mlp$ and $\gb$.
\begin{prop}\emph{\cite{LThesis}} \label{prop:effMp}
If $G$ has $\gamma(G) = \gamma_b(G)$ and does not have an efficient $\gamma$-set, then $\gb-\mlp \geq 1$.  
\end{prop}

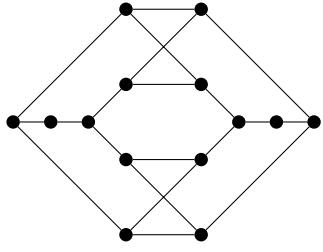
\begin{figure}[H]
			\centering
				\begin{tikzpicture}								
				\node[std](e1) at (0,1.5) {};
				\node[std](e2) at (4,1.5) {};				
				\node[std](m1) at (0.5,1.5) {};
				\node[std](m2) at (3.5,1.5) {};				
				\node[std](h1) at (1,1.5){};
				\node[std](h2) at (3,1.5){};				
				\node[std](x1) at (1.5,3){};
				\node[std](x2) at (2.5,3){};
				\node[std](x3) at (1.5,2){};
				\node[std](x4) at (2.5,2){};				
				\node[std](y1) at (1.5,1){};
				\node[std](y2) at (2.5,1){};
				\node[std](y3) at (1.5,0){};
				\node[std](y4) at (2.5,0){};
				\draw(e1)--(m1)--(h1)--(x3);
				\draw(e2)--(m2)--(h2)--(x4);
				\draw (h1)--(y1);
				\draw(h2)--(y2);
				\draw(x1)--(x2)--(x3)--(x4)--(x1);
				\draw(y1)--(y2)--(y3)--(y4)--(y1);
				\draw(e1)--(x1);
				\draw(e2)--(x2);
				\draw(e1)--(y3);
				\draw(e2)--(y4);
			\end{tikzpicture}
				\caption{A graph $G$ with $\gamma_b(G) = 4$, $\mlp(G) = 2$.}%
	\label{fig:ceMP}
\end{figure}

\noindent The graph in Figure \ref{fig:ceMP} (see \cite{LThesis}) was the first known example where $\gamma_b - \mlp > 1$.  The following year, Hartnell and Mynhardt \cite{HartMyn} provided a construction for a graph $G_k$ with $\gamma_b(G_k) - \mlp(G_k) = k$ for any $k \geq 1$, thereby demonstrating that the $\gamma_b - \mlp$ difference can be arbitrary.

\begin{figure}[H]
			\centering
				\begin{tikzpicture}
				\foreach \i/\x in {1/2.5, 2/7.5, 3/12.5}
							\node [std](a\i) at (\x,0) {};	
				\foreach \i/\x in {1/1, 2/2, 3/3, 4/4, 5/6, 6/7, 7/8, 8/9, 9/11, 10/12, 11/13, 12/14 }
							\node [std](b\i) at (\x,1) {};								
				\foreach \i/\x in {1/2.5, 2/7.5, 3/12.5}
							\node [std](c\i) at (\x,2) {};								
			
			\draw(c1)--(b1)--(a1)--(b2)--(c1)--(b3)--(a1)--(b4)--(c1);	
			\draw(c2)--(b5)--(a2)--(b6)--(c2)--(b7)--(a2)--(b8)--(c2);	
			\draw(c3)--(b9)--(a3)--(b10)--(c3)--(b11)--(a3)--(b12)--(c3);	
			\draw(b4)--(b5);
			\draw(b8)--(b9);			
		
				\node[above=1mm of c1][color=black]{$c_1$}; 
				\node[above=1mm of b1][color=black]{$a_1$};
				\node[above=1mm of b4][color=black]{$b_1$}; 
				\node[below=1mm of a1][color=black]{$u_1$};		
				\node  at (1.7,1) {$r_1$};	
				\node  at (3.3,1) {$s_1$};						
				
				\node[above=1mm of c2][color=black]{$c_2$}; 
				\node[above=1mm of b5][color=black]{$a_2$};
				\node[above=1mm of b8][color=black]{$b_2$}; 
				\node[below=1mm of a2][color=black]{$u_2$};		
				\node  at (6.7,1) {$r_2$};	
				\node  at (8.3,1) {$s_2$};					
		
				\node[above=1mm of c3][color=black]{$c_3$}; 
				\node[above=1mm of b9][color=black]{$a_3$};
				\node[above=1mm of b12][color=black]{$b_3$}; 
				\node[below=1mm of a3][color=black]{$u_3$};		
				\node  at (11.7,1) {$r_3$};	
				\node  at (13.3,1) {$s_3$};	
				\node[bred] at (c2) {};
				\node[mp] at (a1) {};
				\node[mp] at (a2) {};
				\node[mp] at (a3) {};
				
				\node[right=1mm of c2][color=red]{$4$};
			\end{tikzpicture}	
			\caption{The graph $G_1$ with $\gamma_b(G_1)=4$ and $\mlp(G_1)=3$}
	\label{fig:HartMynMPG1}
\end{figure}
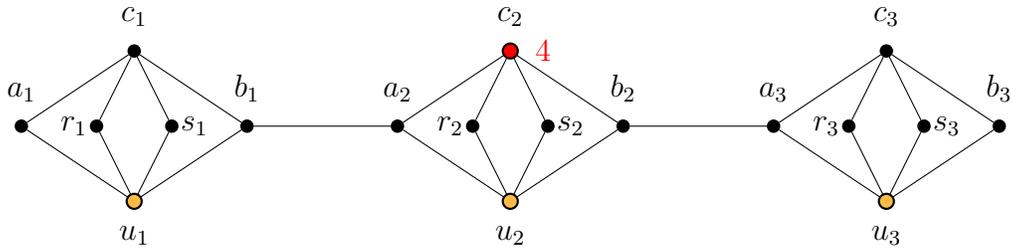

\bigskip

\noindent\textbf{Construction of $G_k$}  \cite{HartMyn}: For $i=1,2,\dots,3k$, let $H_i \cong K_{2,4}$ with bipartition $(X_i,Y_i)$, where $X_i=\{c_i, u_i\}$ and $Y_i=\{a_i,r_i,s_i,b_i\}$.  Form $G_k$ by joining $b_i$ and $a_{i+1}$ for each $i=1,2,\dots, 3k-1$.  

\bigskip

The graph $G_1$ is illustrated in Figure \ref{fig:HartMynMPG1}.  Notice that each induced subgraph $H_i$ contains at most one multipacking vertex, and so $\mlp(G_1) \leq 3$.  Furthermore, recall from Proposition \ref{prop:mpDiam} that 
\begin{align*}
\mlp(G_1) \geq \left\lceil \frac{\diam(G_1)+1}{3}\right\rceil = \left\lceil \frac{8+1}{3}\right\rceil = 3.
\end{align*}
Thus the yellow vertex set (when viewed in colour)  $U=\{u_1,u_2,u_3\}$ forms a maximum multipacking on  $G_1$.  The generalized graph $G_k$ is illustrated in Figure \ref{fig:HartMynMPb}.  The depicted broadcast $f$ with
\begin{align*}
	f(v) &= \begin{cases}
		4 & \text{ if } v=c_i \text{ and } i \equiv 2 \ (\mod 3)\\
		0 & \text{ otherwise,}
	\end{cases}
\end{align*}
is clearly a dominating broadcast; we will show that $f$ is a minimum cost broadcast using Brewster and Duschesne's application of fractional multipackings in Section \ref{sec:fractional}.

\begin{figure}
			\centering
				\begin{tikzpicture}
				\foreach \i/\x in {1/0.75, 2/3.25, 3/5.75}
							\node [bblue](a\i) at (\x,0) {};	
				\foreach \i/\x in {1/0, 2/0.5, 3/1, 4/1.5, 5/2.5, 6/3, 7/3.5, 8/4, 9/5, 10/5.5, 11/6, 12/6.5 }
							\node [bblue](b\i) at (\x,1) {};								
				\foreach \i/\x in {1/0.75, 2/3.25, 3/5.75}
							\node [bblue](c\i) at (\x,2) {};								
			
			\draw(c1)--(b1)--(a1)--(b2)--(c1)--(b3)--(a1)--(b4)--(c1);	
			\draw(c2)--(b5)--(a2)--(b6)--(c2)--(b7)--(a2)--(b8)--(c2);	
			\draw(c3)--(b9)--(a3)--(b10)--(c3)--(b11)--(a3)--(b12)--(c3);	
			\draw(b4)--(b5);
			\draw(b8)--(b9);			
				\foreach \i/\x in {1/8.75, 2/11.25, 3/13.75}
							\node [smred](d\i) at (\x,0) {};	
				\foreach \i/\x in {1/8, 2/8.5, 3/9, 4/9.5, 5/10.5, 6/11, 7/11.5, 8/12, 9/13, 10/13.5, 11/14, 12/14.5 }
							\node [smred](e\i) at (\x,1) {};	
				\foreach \i/\x in {1/8.75, 2/11.25, 3/13.75}
							\node [smred](f\i) at (\x,2) {};	
			\draw(f1)--(e1)--(d1)--(e2)--(f1)--(e3)--(d1)--(e4)--(f1);	
			\draw(f2)--(e5)--(d2)--(e6)--(f2)--(e7)--(d2)--(e8)--(f2);	
			\draw(f3)--(e9)--(d3)--(e10)--(f3)--(e11)--(d3)--(e12)--(f3);	
			\draw(e4)--(e5);
			\draw(e8)--(e9);			
			
			\node (h1) at (7,1) {};	
			\node (h2) at (7.5,1) {};	
			
			\draw(b12)--(h1);
			\draw(e1)--(h2);
			\draw[tr] (b12)--(e1);
				\node[right=1mm of f2][color=red]{$4$}; 
				\node[right=1mm of c2][color=blue]{$4$};
				 
				 \node[above=1mm of b1][color=black]{$a_1$};			
				\node[above=1mm of b4][color=black]{$b_1$}; 
				 
				\node[above=1mm of b9][color=black]{$a_3$};			
				\node[above=1mm of b12][color=black]{$b_3$}; 
				
				\node[above=1mm of e9][color=black]{$a_{3k}$};			
				\node[above=1mm of e12][color=black]{$b_{3k}$}; 

				\node[above  left=1mm and 1mm of e2][color=black]{$a_{3k-2}$};			
				\node[above right=1mm and 1mm of e3][color=black]{$b_{3k-2}$};

			\end{tikzpicture}	
				\caption{The graph $G_k$ with $\gamma_b(G_k)=4k$ and $\mlp(G_k)=3k$}
	\label{fig:HartMynMPb}
\end{figure}
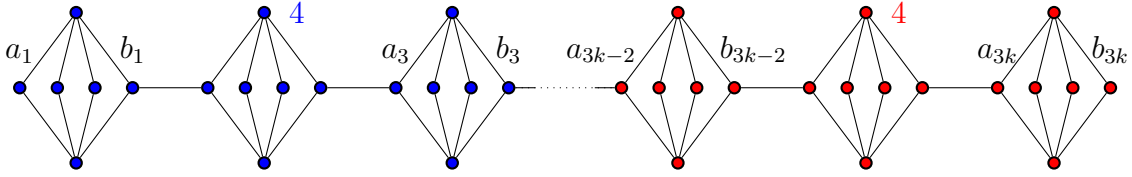

The graph $G$ in Figure \ref{fig:ceMP} has ratio $\gamma_b(G) /  \mlp(G) = 2$, while the graph $G_k$ has \linebreak $\gamma_b(G_k) /  \mlp(G_k) = 4 / 3$.  There are currently no known graphs with ratio $\gamma_b /  \mlp > 2$.  Hartnell and Mynhardt note that if a graph $H$ has $\mlp(H)=1$ or $2$, then $\gamma_b(H)  /  \mlp(H) \leq 2$.  Thus, if $G$ has $\gamma_b(G) /  \mlp(G) > 2$, then $\mlp(G) \geq 3$.  By Proposition \ref{prop:mpB},  if $G$ has $\mlp(G)=3$, then $\gamma_b(G) \leq 7$.  It follows that if such a graph $G$ with $\mlp(G)=3$ exists, then $G$ has $\gamma_b(G)=7$.  Hartnell and Mynhardt were unable to construct such an extremal graph; however, to aid in future investigation they formulated a series of seven structural facts.

\noindent \textbf{Facts }\cite{HartMyn}: Suppose that $G$ is a connected graph with $\mlp(G)=3$ and $\gamma_b(G)=7$.
Let $\alpha$ be a peripheral vertex of $G$.  For $i=0,1,\dots,8$, let $V_i=\{x \in V(G): d(x,\alpha)=i\}$.
\begin{enumerate}
	\item[$(i)$] For all vertices $u,v \in V_3 \cup V_4$, $d(u,v) \leq 4$.
	\item[$(ii)$] For any $u \in V_3\cup V_4$, any $y\in V_5\cup \dots \cup V_8$ and any $w \in V_7 \cup V_8$, $d(u,y) \leq 4$ or $d(y,w) \leq 2$.
	\item[$(iii)$] If $\diam(G)=8$, then for any $u\in V_3\cup V_4, V_5 \subseteq N_4[u]$.
	\item[$(iv)$] For each $u\in V_3 \cup V_4$, there exists $v\in V_4$ such that $d(u,v) \geq 3$.  Furthermore, if $\diam(G)=7$, there is such a $v$ that $d(u,v)=4$.
	\item[$(v)$] Consider any $u,v \in V_3 \cup V_4$ such that $3 \leq d(u,v) \leq 4$.
		\begin{enumerate}
			\item[$(a)$] If $\diam(G)=8$ and $\{u,v\} \cap V_3 \neq \varnothing$, then there exists a $u-v$ path of length at most four that contains a vertex in $V_2$.
			\item[$(b)$] If $\diam(G)=8$ and $\{u,v\} \subseteq V_4$, then there exists a $u-v$ path of length at most four that contains a vertex in $V_2$, or such a path that contains a vertex in $V_6$.
			\item[$(c)$] If $\diam(G)=7$ then there exist a $u-v$ path of length at most four that contains a vertex in $V_2$, or such a path that contains a vertex $V_5$. 
		\end{enumerate}
	\item[$(vi)$] For each $u \in V_3$ there exists $v \in V_2$ such that $d(u,v)=5$.  There also exists a path from $v$  to $V_3$ that does not contain $u$.
	\item[$(vii)$] For each $u \in V_6 \cup V_7$ and each $w \in V_7$, $d(u,w) \leq 4$.  Moreover, if $\diam(G)=8$, then for each $u\in V_6 \cup V_7 \cup V_8$ and each $w \in V_8$, $d(u,w) \leq 2$.
\end{enumerate}


\section{Fractional Broadcasts and Multipackings}
\label{sec:fractional}

Fractional relaxations are a natural extension of the broadcast and multipacking IP's.  The primal problem, the \emph{fractional broadcast primal linear program} B-PLP, finds the minimum cost \emph{fractional broadcast} of a graph $G$ with \emph{fractional broadcast number} $\gamma_{b,f}(G)$.
Now the fractional value of $x_{k,v}$ can be viewed as the intensity or perhaps quality of the signal.  For example $x_{k,v} = 1 \slash 2$ represents that half of a full signal is broadcast from vertex $v$ to all vertices at distance at most $k$ away.  For a vertex $v$ to be dominated by a fractional broadcast, the sum of the signal intensities heard by $v$ must sum to at least one.    
  
\hspace{5mm}\textbf{The Fractional  Broadcast PLP: $B-PLP(G):$} 
\begin{align*}
\min & \ c \cdot x \\
\text{s.t.}&\ Ax \geq \vec{1}\\
 &x_{k,v}\geq 0.
\end{align*}%

For the \emph{fractional multipacking dual linear program} MP-PLP, we view fractional multipackings as a weighting of the vertices rather a vertex subset.  For a graph $G=(V,E)$,  let $y_i$ be the multipacking weight of vertex $v_i$ and define $y$ to be the row vector with entires $y_i$. 

\hspace{5mm}\textbf{The Fractional Multipacking DLP: $MP-DLP(G):$} 
\begin{align*}
\max &\ y \cdot \vec{1} \\
\text{s.t.}&\ yA \leq c \\
& y_{u}\geq 0.
\end{align*}%

Again, by the strong duality theorem for linear programming, 
\begin{align}\label{eq:frac}
 \mlp \leq \mlp_f = \gamma_{b,f} \leq \gb. 
\end{align}

Fractional broadcasts have yet to be studied in any detail; however, by (\ref{eq:frac}) it is possible that they will be useful in later investigation in determining which graphs have $\mlp=\gamma_b$.  In the next section, we present some early results and applications of fractional multipackings, as investigated by Brewster and Duchesne \cite{BD}.

\subsection{Applications of Fractional Multipackings}
\label{subsec:fracApps}

Similar to the way multipackings can be used to certify the minimality of a given dominating broadcast, fractional multipackings can also be used as a certification tool in graphs where $\mlp < \gb$.  For example, recall the graph $G_k$ defined by Hartnell and Mynhardt in \cite{HartMyn} and pictured in Figure \ref{fig:HartMynMPb} of Section \ref{sec:broadcast}.  The given broadcast $f$ with $f(c_i)=4$ for $i\equiv 2 (\mod 3)$, and $f(v)=0$ otherwise is clearly dominating, but showing that $f$ is a minimum cost dominating broadcast is not immediate.  The original proof by Hartnell and Mynhardt was fairly technical; Brewster and Duchesne offer a clever alternative.  

\begin{prop} \emph{\cite{HartMyn}}
The graph $G_k$ in Figure \ref{fig:HartMynMPb} has $\gamma_b(G_k) = 4k$.
\end{prop}

\noindent\textbf{Proof.} Define a fractional multipacking $y$ on $G_k$ such that for each $i\in\{1,2,\dots, 3k\}$,   $y_{r_i}=y_{s_i}=y_{c_i}=y_{u_i}=1\slash 3$ and $y_{a_i}=y_{b_i}=0$.  
It is easy to see that for any vertex $v\in V(G_k)$ and  $1 \leq \ell \leq 4 $,
\begin{align*}
	\sum_{u\in N_{\ell}[v]} y_u \leq \ell.
\end{align*}
Furthermore  for $\ell \geq 5$,
\begin{align*}
	\sum_{u\in N_{\ell}[v]} y_u \leq \sum_{u\in N_{\ell -3}[v]} y_u + \frac{8}{3} \leq \ell. 
\end{align*}
Therefore, $y$ is a feasible fractional multipacking and by strong duality,
\begin{align*}
		\mlp_f(G_k) = y\cdot \vec{1} = 4k = \gamma_b(G).
\end{align*}
\hfill \rule{0.5em}{0.5em}

\noindent This immediately gives the following nice corollary.  

\begin{coro} \emph{\cite{BD}} The $\mlp_f - \mlp$ difference can be arbitrarily large.  The integrality gap is at most $3 \slash 4$.
\end{coro}

Brewster and Duchesne also investigated fractional multipackings on vertex transitive graphs.  Let $G$ be a vertex transitive graph and $v$ any vertex of $G$.  Now let 
\begin{align*}
	w_v(r) = \frac{r}{|N_r[v]|}.
\end{align*}
Since $G$ is vertex transitive, $w_v(r)$ is the same for each vertex; let $w(r)$ be this common value.  Define
\begin{align*}
	w^* = \min_{1 \leq r \leq \rad(G)} w(r)
\end{align*}
and let $r^*$ be the value of $r$ such that $w^*=w(r^*)$.  For each $r^*$-neighbourhood of any vertex $v$, let $n=r^* \slash w^*$ be the number of vertices in $N_{r^*}[v]$.  By symmetry, if $u \in N_{r^*}[v]$, then $v\in N_{r^*}[u]$; therefore, each vertex belongs to exactly $n$ $r^*$-neighbourhoods.

\begin{theorem}\label{thm:vertexTrans} \emph{\cite{BD}}
For a vertex transitive graph $G = (V,E)$, let $y_v = w^*$ for all $v\in V$.  Then $y$ is a maximum fractional multipacking.
\end{theorem}

\begin{proof}
	Fix any $u \in V$ and consider $N_r[u]$ for some $1 \leq r \leq \rad(G)$.  Then, since
	\begin{align*}
		\sum_{v \in N_r[u]}y_v = w^* \cdot |N_r[u]| =  \frac{r^*}{|N_r[u]|}\cdot |N_r[u]| \leq  r,
	\end{align*}
	$y$ is feasible.  Now suppose that $y'$ is any other feasible fractional multipacking of $G$.  Then
	\begin{align*}
		y\cdot \vec{1} &= \sum_{u\in V} w^* = \sum_{u\in V} \frac{r^*}{n} 
		 \geq \sum_{u\in V}\frac{1}{n} \sum_{v \in N_{r^*}[u]}y_v' 
		= \sum_{v\in V}y_v'\left(\sum_{u \in N_{r^*}[v]}\frac{1}{n} \right)
		=  \sum_{v\in V}y_v'
		= y'\cdot \vec{1}.
	\end{align*} 
	Thus $y\cdot \vec{1} \geq y'\cdot \vec{1}$ for any other fractional multipacking $y'$, and  therefore $y$ is a maximum fractional multipacking.
\end{proof}

For example, consider the Petersen graph $P$.  For any vertex $v \in V(P)$, 
\begin{align*}
w(1)= \frac{1}{4} > \frac{2}{10} = \frac{1}{5}= w(2).
\end{align*}
Thus for the Petersen graph, $w^* = 1 \slash 5$, and $\mlp_f(P) = 2$. Notice that since  $\diam(P)=\rad(P) =2$, $\mlp(P)=1$ and $\gamma_b(P)=2$.  Therefore, the Petersen graph provides another  example where $	\mlp(P) < \mlp_f(P) = \gamma_b(G)$.

This notion of spreading the minimum fractional multipacking weight around a graph can also be used  as a lower bound of $\mlp_f$ for general graphs.  For a graph $G$, let 
\begin{align*}
w^* = \min_{v\in V, \ 1 \leq r \leq e(v)} w_v(r)
\end{align*}
and again set $y_v=w^*$ for all $v \in V$.  Then $|V|\cdot w^*$ is a trivial lower bound for $\mlp_f$.

\section{Farber's Algorithm}
\label{sec:farber}

In this section we examine a new perspective on broadcasts and multipackings currently being developed by Brewster and Duchesne in \cite{BD}, and Brewster, MacGillivray and Yang in \cite{BMY}.  This research provides an exciting amalgamation  between the very young concept in multipackings and an older algorithm developed by Martin Farber in the early 1980's.  Since strongly chordal graphs play a pivotal role in the use of Farber's algorithm, we begin this section with an introduction to strongly chordal graphs and some of  their characterizations. 

\subsection{Strongly Chordal Graphs} \label{subsec:strChord}

A graph is \emph{chordal} (or \emph{triangulated}) if it does not contain an induced cycle of length greater than three.  The class of chordal graphs contains many famous families including trees, threshold graphs, interval graphs, split graphs, and maximal outerplanar graphs.  A graph $G=(V,E)$ is said to have a \emph{perfect elimination ordering} if its vertices can be ordered $v_1, v_2,\dots, v_n$ such that for each $i,j$ and $\ell$, if $i<j, i<\ell$ and $v_{\ell}, v_j \in N[v_i]$, then $v_{\ell} \in N[v_j]$.  In \cite{Rose}, Rose showed that a graph is chordal if and only if it has a perfect elimination ordering.

Farber \cite{F83} strengthened this condition by defining a \emph{strong elimination ordering}.  A strong elimination ordering of a graph $G=(V,E)$ is a vertex ordering $v_1, v_2,\dots, v_n$ such  that for each $i,j,k$ and $\ell$, if  $i<j, k<\ell, v_k, v_{\ell} \in N[v_i]$ and $v_k \in N[v_j]$, then $v_{\ell} \in N[v_j]$.  Farber defined a graph to be \emph{strongly chordal} if it admits a strong elimination ordering.   We examine three of Farber's characterizations of strongly chordal graphs.

\begin{theorem}\emph{\cite{F83}} \label{thm:simplechord}
A graph is strongly chordal if and only if every induced subgraph of $G$ contains some vertex $v$ such that for each $u,w \in N[v]$, $N[u] \subset N[w]$ or vice versa; that is, $G$ has a \emph{simple} vertex $v$.
\end{theorem}

To prove this result, Farber developed Algorithm \ref{alg:simpleV} below (different from this section's namesake) that when given any graph $G$ as input will either find a strong elimination ordering of $G$ or locate an induced subgraph of $G$ with no simple vertex.  This algorithm is useful in its own right, as some algorithms (e.g. Algorithms \ref{alg:farb} and \ref{alg:BDfarb}) require a strong elimination ordering as input.

\begin{algo}\label{alg:simpleV}\emph{\cite{F83}} \ 
\begin{enumerate}[leftmargin=1in]
	\item[\emph{\textbf{Input:}}]  A graph $G=(V,E)$.
	\item[\emph{\textbf{Output:}}] A strong elimination ordering or an induced subgraph without a simple vertex.
	\item[\emph{\textbf{Initial:}}] Set $n \rightarrow |V|$.
	\item[\emph{\textbf{Step 1:}}]  Let $V_0=V$ and let $(V_0, <_0)$ be a partial ordering on $V_0$ with $v<_0 u $ if and only if $v=u$.  Let $V_1=V$ and set  $i\leftarrow 1$.
	\item[\emph{\textbf{Step 2:}}]  	Let $G_i$ be the subgraph of $G$ induced by $V_i$.  If $G_i$ has no simple vertex, \emph{OUTPUT} $G_i$ and \emph{STOP}.  Otherwise, define an ordering on $V_i$ by $v <_i u$ if $v<_{i-1} u$ or $N_i[v] \subsetneq N_i[u]$.  
		\item[\emph{\textbf{Step 3:}}] Choose  a $v_i$ which is simple in $G_i$ and minimal in $(V_i, <_i)$.  Let $V_{i+1}=V_i-\{v_i\}$.  If $i=n$ \emph{OUTPUT} ordering $v_1,v_2,\dots,v_n$ of $V$ and \emph{STOP}.  Otherwise, set $i\leftarrow i+1$ and \emph{GO TO  Step 2}.
\end{enumerate} 
\end{algo}

Farber also determined a forbidden subgraph characterization for strongly chordal graphs. A \emph{trampoline} is a split graph $G$ on $2n$ vertices for $n \geq 3$, with vertex partitions $W=\{w_1,w_2,\dots, w_n\}$ and $U=\{u_1,u_2,\dots, u_n\}$, where $W$ is independent, $G[U] \cong K_{|U|}$ and for each $i$ and $j$, $u_jw_i \in E(G)$ if and only if $i=j$ or $i\equiv j+1 (\mod n)$.
The trampolines on four and six vertices are illustrated in Figure \ref{fig:tramps} with the vertices of $U$ in blue and the vertices of $W$ in red, when viewed in colour.

\begin{figure}[h]
			\centering
				\begin{tikzpicture}				
					\node [bblue](u1) at (1,0) {};
					\node [bblue](u2) at (0.5,0.8) {};
					\node [bblue](u3) at (1.5,0.8) {};
					
					\node [smred](w1) at (0,0) {};
					\node [smred](w2) at (2,0) {};
					\node [smred](w3) at (1,1.7) {};					
			
				\draw(u1)--(w1)--(u2)--(w3)--(u3)--(w2)--(u1);
				\draw (u1)--(u2)--(u3)--(u1);			
			\end{tikzpicture}	
					\hspace{1cm}
			\begin{tikzpicture}				
					\node [bblue](u1) at (0.5,0.5) {};
					\node [bblue](u2) at (1.5,0.5) {};
					\node [bblue](u3) at (1.5,1.5) {};
					\node [bblue](u4) at (0.5,1.5) {};
					
					\node [smred](w1) at (1,0) {};
					\node [smred](w2) at (2,1) {};
					\node [smred](w3) at (1,2) {};	
					\node [smred](w4) at (0,1) {};					
			
				\draw(u1)--(w1)--(u2)--(w2)--(u3)--(w3)--(u4)--(w4)--(u1);
				\draw (u1)--(u2)--(u3)--(u4)--(u1)--(u3)--(u4)--(u2);			
			\end{tikzpicture}				
				\caption{Trampolines with $n=2$ and $n=3$.}
	\label{fig:tramps}
\end{figure}
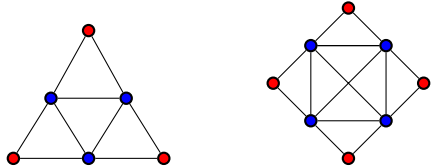

\begin{theorem}\emph{\cite{F83}} \label{thm:tramp}
A chordal graph is strongly chordal if and only if it contains no induced trampoline subgraph.
\end{theorem}

The final characterization we present here is the most relevant to the study of multipackings, and graph optimization problems in general.  For a graph $G$ with $V(G)=\{v_1,v_2,\dots,v_n\}$, the \emph{neighbourhood matrix} $M(G)$ is the $n \times n$ vertex-closed neighbourhood incidence matrix of $G$ with $m_{ij}=1$ if $v_i \in N[v_j]$ and $m_{ij}=0$ otherwise.

\begin{prop}\emph{\cite{F83}}\label{prop:strElimGam}
For a graph $G$, the ordering $v_1,v_2,\dots, v_n$ of its vertices is a strong elimination ordering if and only if the \emph{$\Gamma$-matrix},
\begin{align*}
\Gamma = \begin{bmatrix}
1&1\\
1&0
\end{bmatrix},
\end{align*}
is not a submatrix of the neighbourhood matrix $M(G)$.
\end{prop}
This result is immediate by the definition of strong elimination ordering.  Thus $G$ is strongly chordal if and only if $M(G)$ is $\Gamma$-free.  A $(0,1)$-matrix is \emph{totally balanced} if it does not contain an incidence matrix of any cycle of length at least three as a submatrix.  

\begin{theorem}\emph{\cite{F83}}\label{thm:strChordTB}
A graph $G$ is strongly chordal if and only if $M(G)$ is totally balanced. 
\end{theorem}
Farber's proof is immediate by Proposition \ref{prop:strElimGam}, observing that the $\Gamma$-matrix is an edge-vertex submatrix of every cycle of length at least three.  


\subsection{Farber's Primal-Dual Algorithm} \label{subsec:FarbAlg}

Farber's original algorithm  \cite{Farber} is a linear-time search developed to find a minimum weight dominating set of a vertex subset of a strongly chordal graph.    Following his notation, the weight of each vertex $v_i$ is denoted  $w_i$.  Since the problem of finding a minimum weight dominating set with arbitrary real weights can be reduced to the problem with real positive weights, we proceed with the assumption that $w_i > 0$ for all $i$.  Furthermore, for vertices $v_i$ and $v_j$, if $i=j$ or $v_iv_j \in E(G)$, we write $i \sim j$.  The definition of strong elimination ordering can be altered to utilize this new notation.

\begin{lemma}\label{lem:strElim} \emph{\cite{Farber}} For a graph $G$, an ordering $v_1, v_2,\dots, v_n$ of its vertices is a strong elimination ordering if and only if for each $i,j,k,\ell$, with $i\leq j$ and $k\leq \ell$, and $i\sim k, i \sim \ell$ and $j \sim k$, then $j\sim \ell$.	
\end{lemma}

We present a slight modification to Farber's original LP problems here; for simplicity, in this paper we consider only the problem of finding a weighted dominating set of the entire graph, not just a vertex subset.

\hspace{5mm}\textbf{The Primal $P(G):$} \vspace{-5mm}
\begin{align*}
\min \hspace{3mm}& \sum_{i=1}^{n}w_ix_i \\
\hspace{5mm}\text{s.t.}& \sum_{i\sim j} x_i \geq 1 \text{ for each } j,\\
& x_i\geq 0 \text{ for each } i.
\end{align*}

\hspace{5mm}\textbf{The Dual $D(G):$} \vspace{-5mm}
\begin{align*}
\max \hspace{3mm}& \sum_{j=1}^{n}y_j \\
\hspace{5mm}\text{s.t.}& \sum_{j\sim i} y_j \leq w_i \text{ for each } i,\\
& y_j\geq 0 \text{ for each } j.
\end{align*}

Farber's algorithm, as shown in Algorithm \ref{alg:farb}, solves the above LP $P(G)$ and its dual $D(G)$.  When each vertex $v_i$ of a graph $G$ has weight $w_i=1$, this is equivalent to finding a $\gamma$-set in the primal problem and a $\rho$-set in the dual problem. 

Algorithm \ref{alg:farb} executes in two stages.  In the first stage, it finds an optimal solution to the dual problem $D(G)$ by scanning the vertices in the given strong elimination order $v_1, v_2, \dots, v_n$. Here, it greedily adds vertices to the 2-packing by considering the available slackness in each associated neighbourhood of each vertex. In the second stage the algorithm finds an optimal solution to the primal $P(G)$ by scanning the closed neighbourhoods of vertices in reverse strong elimination order $N[v_n], N[v_{n-1}],\dots, N[v_1]$. This finds a dominating set by examining vertices whose neighborhoods have no remaining slack after Stage~1.  

The set $T$ is used to assure that the complementary slackness conditions are satisfied as the algorithm moves through Stage 2.  Farber also defines 

\begin{align}\label{eq:h(i)} h(i)=w_i - \sum_{j\sim i}y_j \end{align}
and 
\begin{align}T_i=\{k: k\sim i \text{ and } y_k>0 \}\end{align}  
to track available slackness.  When the algorithm begins, $T=\{1,2,\dots n\}$, $x_i = 0$ and $y_j=0$; the 2-packing and dominating sets are empty, and thus every vertex has slack.  

\begin{algo}\label{alg:farb}\emph{\cite{BD}} \ 
\begin{enumerate}[leftmargin=1in]
	\item[\emph{\textbf{Input:}}]  A strongly chordal graph $G=(V,E)$ with strong elimination ordering \linebreak $v_1,v_2,\dots, v_n$ and positive vertex weights $w_1, w_2,\dots, w_n$.
	\item[\emph{\textbf{Output:}}] Optimal solutions to $P(G)$ and $D(G)$.
	\item[\emph{\textbf{Initial:}}] Set $T=\{1,2,\dots,n\}$ and each $y_j=0, x_i=0$.
	\item[\emph{\textbf{Step 1:}}]  For each $j=1,\dots, n$, set $y_j \leftarrow \min\{h(k):k \sim j\}$.
	\item[\emph{\textbf{Step 2:}}]  For each $i=n, \dots, 1$, if $h(i)=0$ and $T_i \subset T$, then set $x_i\leftarrow 1$ and $T \leftarrow T-T_i$.	
\end{enumerate} 
\end{algo}

The algorithm clearly halts in $O(2n)$ operations.  Step 1 ensures that $y_j \geq 0$ and $h(j) \geq0$ for $j$, and therefore the solution presented by Algorithm \ref{alg:farb} is dual feasible.  Furthermore, for each $i$, $x_i \in \{0,1\}$.  Thus to show the feasibility of the primal solution, it suffices to show that for each $j$, there is some $i \sim j$ with $x_i=1$.   Since $y_j$ is the $\min\{h(k):k \sim j\}$, there exists some $k \sim j$ such that $h(k)=0$ and $\max T_k \leq j$.  If $x_k=1$, we are done.  Otherwise, if $x_k=0$, when the algorithm scanned $v_k$ in Stage 2 $T_k$ was not a subset of $T$.  Since the vertices in Stage 2 are scanned in descending index order, this implies that there is some $\ell > k$ such that $x_{\ell}=1$ and $T_{\ell}\cap T_k \neq \varnothing$.  Let $i \in T_{\ell} \cap T_k$, and then by transitivity of $\sim$, $\ell \sim i \sim k \sim j$.  It follows that $i \leq j$ since $\max T_k \leq j$.   The vertices were presented in a strong elimination order, and so by  Lemma \ref{lem:strElim} it follows that $\ell \sim j$.    Therefore, there exists some $\ell \sim j$ with $x_{\ell} =1$ as required.  This demonstrates the feasibility of the primal solution.

To confirm that the solutions are optimal, suppose that some $x_i > 0$; that is, $x_i=1$.  Then there is no slackness available around $v_i$, so $h(i)=0$.  By (\ref{eq:h(i)}), 
\begin{align*}
\sum_{j\sim i} y_j = w_i.
\end{align*}
Now suppose $y_j >0$.  It follows that 
$\sum_{i\sim j} x_i \leq 1$
because the algorithm requires that if $x_i=x_k=1$, then $T_i \cap T_k = \varnothing$.  Combining this with the feasibility requirement that $\sum_{i\sim j} x_i \geq 1$ yields that
\begin{align*}
\sum_{i\sim j} x_i =1.
\end{align*}
Thus both the primal and dual solution are tight and therefore optimal.

See Example  \ref{ex:farbOrig} in  Appendix \ref{ap:farber} for an example of Algorithm \ref{alg:farb} applied to a tree.


\subsection{Extension to Broadcasts and Multipackings} \label{subsec:FarbBroadMP}

Recently, Brewster and Duchesne \cite{BD} extended Farber's original algorithm (Algorithm \ref{alg:farb}) to broadcasts and multipackings.  The original primal solution from the algorithm provided a minimum weight dominating set.  To account for the farther reaching nature of broadcasts, Brewster and Duchesne modified the algorithm to instead search for a minimum weight covering of $k$-neighbourhoods (or balls), where each $k$-neighbourhood had weight (or cost)~$k$.

As Brewster and Duchesne's paper is still being drafted, we have taken some liberties in guessing applicable notation and the exact formulation of Algorithm \ref{alg:BDfarb}.  Recall from Section \ref{sec:intro} that for a graph $G$ with $V(G)=\{v_1,v_2,\dots,v_n\}$, the extended neighbourhood matrix $A$ is the $n \times m$ vertex- multi-neighbourhood incidence matrix of $G$.  The $n$ rows are indexed by the vertices $v_i$ and the $m$ columns are indexed as pairs $(j,k)$ to denote the $k$-neighbourhood of vertex $v_j$.  The entries of $A$ are such that 
\begin{align*} a_{i,(j,k)} =
	\begin{cases}
1 & \text{ if } v_i \in N_k[v_{j}], \\
0 & \text{ otherwise.}		
	\end{cases}
\end{align*}
For convenience, we also extend some of Farber's notation.  We write $ i \sim_m (j,k)$ if either $i=j$ or $v_i \in N_k[v_j]$.

Notice that the weights are assigned to vertex neighbourhoods, rather than the vertices themselves.  The weight of the $k$-neighbourhood of the vertex $v_i$ is denoted $w_{i,k}=k$. Thus, let
\begin{align}\label{eq:h(i,k)} h(i,k)=w_{i,k} - \sum_{j\sim_m (i,k)}y_j = k- \sum_{j\sim_m (i,k)}y_j, \end{align}
and 
\begin{align}T_{i,k}=\{j: j\sim_m (i,k) \text{ and } y_j>0\}.\end{align}

\begin{algo}\label{alg:BDfarb}\emph{\cite{BD}} \ 
\begin{enumerate}[leftmargin=1in]
	\item[\emph{\textbf{Input:}}]  A weighted graph with $G=(V,E)$ with strong elimination order $v_1, v_2,\dots, v_n$.
	\item[\emph{\textbf{Output:}}] Optimal solutions to  BIP(G) and MIP(G).
	\item[\emph{\textbf{Initial:}}] Set $T=\{1,2,\dots,n\}$ and each $y_j=0$ and $x_{i,k}=0$.
	\item[\emph{\textbf{Step 1:}}]  For each $j=1,\dots, n$, set $y_j \leftarrow \min\{h(i,k): j \sim_m (i,k)\}$.
	\item[\emph{\textbf{Step 2:}}]  For $(i,k)$ in descending lexicographic order, if $h(i,k)=0$ and $T_{i,k} \subset T$, then set $x_{i,k}\leftarrow 1$ and $T \leftarrow T-T_{i,k}$.	
\end{enumerate} 
\end{algo}

For each $v_i$, we could consider each of its $k$-neighbourhoods for $k=1,\dots, e(v)$; however, in an optimal broadcast setting, we can ignore certain neighbourhoods that we know will never be selected in a minimum cost broadcast.  For example, if $v_{\ell}$ is a leaf of a tree with $e(v_{\ell}) \geq 2$, there is no incentive to define a broadcast $f$ with $f(v_{\ell}) =  e(v_{\ell})$; if $v$ is the stem of $v_{\ell}$ we can always cover at least as many vertices with a broadcast $g(v)=e(v_{\ell})$.  Thus, we can safely remove some neighbourhoods from $A$.  It is likely that exactly which neighbourhoods are removable is dependent upon each class of graph.

To apply the algorithm to a tree $T$, Brewster and Duchesne give the following construction of a specifically ordered extended neighbourhood matrix $M$.  We provide an example of Algorithm \ref{alg:BDfarb} being applied to a tree in Example \ref{ex:BDfarb} in Appendix \ref{ap:farber}.

\noindent\textbf{Construction of $M$:} Given any tree $T$, root $T$ at a central vertex.  For each $v \in V(T)$, let $\ell(v)$ be the maximum distance to a leaf below $v$ in $T$.  For each non-leaf vertex $v$, construct a series of balls of radius $1,2,\dots \ell(v)$ centred at $v$.  Define $M$ to be the resulting vertex-ball incidence matrix, with $n$ rows sorted in descending (rooted) level order, and $m$ columns sorted left to right in ascending lexicographic order read from the bottom up.

\begin{prop}\label{prop:GamFreeT} \emph{\cite{BD}}  The resulting matrix $M$ is $\Gamma$-free.
\end{prop}

\begin{proof}
	Suppose to the contrary that $M$ contains a $\Gamma$-submatrix.  Then there exist two balls $B$ and $A$, and two vertices $z$ and $y$ such that $z \in A \cap B$, $y\in B$ and $y\notin A$, as shown in the vertex-expanded neighbourhood submatrix below.  Since the columns of $M$ are sorted lexicographically, there exists some other vertex $x$, such that $x\in A$ but $x\notin B$; otherwise, the column representing $A$ would be to the left of the column representing $B$.
\begin{table}[H] \centering
\begin{tabular}{c|cc}
		& $B$ & $A$ \\
		\hline
$z$	& 1		&	1	\\
$y$	& 1		&	0	\\
$x$	& 0		&	1	\\
\end{tabular} \label{tab:GamTree}
\end{table}

\noindent Let $a$ be the centre vertex of ball $A$, $b$ be the centre vertex of ball $B$, and  $w$ be the least common ancestor of $z$ and $y$.  \\
\noindent \textbf{Case 1:} $w$ is on the $z-x$ path.  Then $d(z,x) = d(z,w)+d(w,x)$.
Since the rows were sorted by decreasing depth, the depth of $z$ in $T$ is at least that of $x$ and $y$. Hence $d(w,y) \leq d(w,z)$.  Notice that $a$ is not the $y-w$ path, since otherwise $d(a,y) \leq d(a,z)$ which implies that $y\in A$.  Suppose that $w$ is on the $a-z$ path.  Then
\begin{align*}
	d(a,y) = d(a,w) + d(w,y) \leq d(a,w) + d(w,z) = d(a,z). 
\end{align*}
Since $z \in  A$ and $d(a,y) \leq d(a,z)$, this implies that $y \in A$, a contradiction.  Thus $a$ and $z$ share the same child of $w$ as an ancestor, and $w$ is on both the $a-x$ and $a-y$ paths.  Recall that $x\in A$ and $y\notin A$, and so $d(a,x)<d(a,y)$, which implies
\begin{align*}
d(a,w) + d(w,x) = d(a,x) < d(a,y) = d(a,w)+d(w,y).
\end{align*}
It follows that $d(w,x) < d(w,y) \leq d(w,z)$.

By a similar argument (substituting $b$ and $x$ for $a$ and $y$, respectively), we conclude that $w$ is not on the $b-z$ path, but is on the $b-x$ and $b-y$ paths. Thus,
\begin{align*}
	d(b,x) = d(b,w)+d(w,x) < d(b,w)+d(w,y) = d(b,y).
\end{align*}
Finally, since $d(b,x)<d(b,y)$ and $y\in B$, it follows that $x\in B$, a contradiction.

\noindent \textbf{Case 2:} $w$ is not on the $z-x$ path.  Let $v$ be the lowest common ancestor of $x$ and $z$.  A similar argument  to Case 1 using $v$ in place of $w$ again implies that $x\in B$ and forms the desired contradiction.
\end{proof}

Forthcoming work by Brewster and Duchesne will demonstrate how this result provides a nice alternative proof to Theorem \ref{th:mpTree}.  Furthermore, Proposition \ref{prop:GamFreeT} demonstrates that both a minimum broadcast and a maximum multipacking can be found in $O(n+m)$ time, where $n$ is the number of vertices and $m$ is the width of the matrix $M$ constructed above Proposition \ref{prop:GamFreeT}.

In a currently unpublished work, Brewster, MacGillivray and Yang \cite{BMY} extend this result to show that the extended neighbourhood matrix of a graph $G$ is $\Gamma$-free if and only if $G$ is strongly chordal.  Although this implies that Algorithm \ref{alg:BDfarb} can only be applied to strongly chordal graphs, this does not complete the class of graphs with $\gamma_b=\mlp$.  Recall for example that $\gb(C_6)=\mlp(C_6)$, but $C_6$ is not chordal.     

\section{Conclusions}
\label{sec:conc}
Having examined some of the main results in the very young study of multipackings and broadcasts in graphs, we conclude our survey with some open problems.

\subsection{Open Problems}

\begin{problem} \emph{\cite{BMT}}  For which graphs $G$ is $\gamma_b(G) = \mlp (G)$?
\end{problem}

\begin{problem}\emph{\cite{HartMyn}} Does there exist a graph $G$ with $\mlp(G) =3$ and $\gamma_b(G)  \geq 5$?  
\end{problem}

\begin{problem}\emph{\cite{HartMyn}} Can the ratio $\gamma_b / \mlp< 3$ be improved for general graphs?
\end{problem}

\begin{problem} Does there exist a graph $G$ with integral $\gamma_{b,f}(G)$ such that \begin{align*}\mlp(G)<\gamma_{b,f}(G) < \gamma_b(G)?\end{align*}
\end{problem}

\begin{problem}
In \cite{Farber}, Farber modified his algorithm to find minimum independent dominating sets of strongly chordal graphs.  Dunbar et al. \cite{DEHHH} define a broadcast $f$ to be \emph{independent} if for each $v \in V_f^+$, $N_f[v] \cap V_f^+ = \{v\}$, or equivalently $|\{u \in V_f^+: d(u,v) \leq f(u)\}|=1$.  Thus if a broadcast is independent, then each broadcast vertex hears only the broadcast from itself.  Can Farber's algorithm for independent dominating sets be extended to independent dominating broadcasts?  Furthermore, what is the dual parameter to the independent broadcast number?
\end{problem}

\begin{problem}
	Can the minimum cost broadcast and maximum multipacking problems be formulated as hypergraph transversal and matching problems? 
\end{problem}

\begin{problem}
	A \emph{clutter} is a hypergraph with no nested edges.  As detailed in \cite{Corn}, clutters have been extensively researched in an optimization context.  In general, the extended neighbourhood hypergraph is not a clutter, since it has many nested edges; however, the hypergraph whose edges are the broadcast neighbourhood of an efficient broadcast is a trivial clutter.  Is there a meaningful way to interpret the minimum cost broadcast or maximum multipacking problem in terms of clutters?
\end{problem}

\appendix
\section{Examples}

\subsection{Farber's Algorithm} \label{ap:farber}

\begin{example}\label{ex:farbOrig} Using Farber's original algorithm (Algorithm \ref{alg:farb}), we find a maximum 2-packing and minimum dominating set for the graph $G$ in Figure \ref{fig:exFarb}.
\begin{figure}[H]
			\centering
				\begin{tikzpicture}
				
					\node [std](a) at (0,0)[label=left:$v_1$]{};
					\node [std](b) at (1,0)[label=right:$v_2$] {};
					\node [std](c) at (2,0.5)[label=right:$v_3$] {};
					\node [std](d) at (0,0.5)[label=left:$v_4$] {};
					\node [std](e) at (1,0.5)[label=right:$v_5$] {};
					\node [std](f) at (2,1)[label=right:$v_6$] {};
					\node [std](g) at (0.5,1)[label=left:$v_7$] {};
					\node [std](h) at (2,1.5)[label=right:$v_8$] {};
					\node [std](i) at (0.5,1.5)[label=left:$v_9$] {};	
					\node [std](j) at (1.25,2)[label=above:$v_{10}$] {};							
			
				\draw(a)--(d)--(g)--(i)--(j)--(h)--(f)--(c);
				\draw(b)--(e)--(g);			
				
			\end{tikzpicture}	
				\caption{A tree $G$ with strong elimination order $v_1,v_2,v_3,v_4,v_5,v_6,v_7,v_8,v_9,v_{10}$.}
	\label{fig:exFarb}
\end{figure}
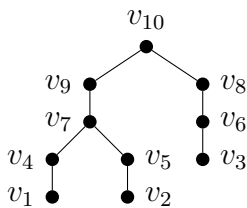

\begin{table}[H]
	\centering
\begin{tabular}{c|cccccccccc}
			&	$N[v_1]$	&	$N[v_2]$	&	$N[v_3]$	&	$N[v_4]$	&	$N[v_5]$	&	$N[v_6]$	&	$N[v_7]$	&	$N[v_8]$	&	$N[v_9]$	&	$N[v_{10}]$	\\

	\hline
$	v_1	$	&	1	&	0	&	0	&	1	&	0	&	0	&	0	&	0	&	0	&	0	\\
$	v_2	$	&	0	&	1	&	0	&	0	&	1	&	0	&	0	&	0	&	0	&	0	\\
$	v_3	$	&	0	&	0	&	1	&	0	&	0	&	1	&	0	&	0	&	0	&	0	\\
$	v_4	$	&	0	&	0	&	0	&	1	&	0	&	0	&	1	&	0	&	0	&	0	\\
$	v_5	$	&	0	&	0	&	0	&	0	&	1	&	0	&	0	&	0	&	0	&	0	\\
$	v_6	$	&	0	&	0	&	0	&	0	&	0	&	1	&	0	&	1	&	0	&	0	\\
$	v_7	$	&	0	&	0	&	0	&	1	&	1	&	0	&	1	&	0	&	1	&	0	\\
$	v_8	$	&	0	&	0	&	0	&	0	&	0	&	1	&	0	&	1	&	0	&	1	\\
$	v_9	$	&	0	&	0	&	0	&	0	&	0	&	0	&	1	&	0	&	1	&	1	\\
$	v_{10}	$	&	0	&	0	&	0	&	0	&	0	&	0	&	0	&	1	&	1	&	1	\\
\end{tabular} \caption{The neighbourhood matrix of $T$}
\end{table}

\begin{enumerate}[leftmargin=1in]

\item[\textbf{Initial:}] $x_v=y_v=0$ for all $v\in V(G)$, $T=\{1,2,3,4,5,6,7,8,9,10\}$.

\begin{table}[H]
	\centering
\begin{tabular}{c|cccccccccc}
$	v_i	$	&	$v_1$	&	$v_2$	&	$v_3$	&	$v_4$	&	$v_5$	&	$v_6$	&	$v_7$	&	$v_8$	&	$v_9$	&	$v_{10}$	\\ \hline
$	h(i)	$	&	1	&	1	&	1	&	1	&	1	&	1	&	1	&	1	&	1	&	1	
\end{tabular} 
\end{table}

\item[\textbf{Stage 1:}] Scan rows in descending order.
\begin{itemize}
 \item $h(i) > 0$ for all $i \sim 1$. 	\\
 			 Update: $y_1=1$.  
				\begin{table}[H] \centering  \begin{tabular}{c|cccccccccc}
				$	v_i	$	&	$v_1$	&	$v_2$	&	$v_3$	&	$v_4$	&	$v_5$	&	$v_6$	&	$v_7$	&	$v_8$	&	$v_9$	&	$v_{10}$	\\ \hline
				$	h(i)	$	&	0	&	1	&	1	&	0	&	1	&	1	&	1	&	1	&	1	&	1	
					\end{tabular}  					\end{table}
					
	 \item  $h(i) > 0$ for all $i  \sim 2$. 	\\ 
	 			Update: $y_2=1$.  
				\begin{table}[H] \centering  \begin{tabular}{c|cccccccccc}
		$	v_i	$	&	$v_1$	&	$v_2$	&	$v_3$	&	$v_4$	&	$v_5$	&	$v_6$	&	$v_7$	&	$v_8$	&	$v_9$	&	$v_{10}$	\\ \hline
				$	h(i)$	&	0	&	0&	1	&	0	&	0	&	1	&	1	&	1	&	1	&	1	
					\end{tabular}  					\end{table}
					
		\item 	 $h(i) > 0$ for all $i \sim 3$. 	\\ 
				Update: $y_3=1$.  
				\begin{table}[H] \centering  \begin{tabular}{c|cccccccccc}
				$	v_i	$	&	$v_1$	&	$v_2$	&	$v_3$	&	$v_4$	&	$v_5$	&	$v_6$	&	$v_7$	&	$v_8$	&	$v_9$	&	$v_{10}$	\\ \hline
				$	h(i)$	&	0	&	0& 	0 &	0	&	0	&	0	&	1	&	1	&	1	&	1	
		\end{tabular}  					\end{table}

		\item $h(4)=0$ and $4\sim 4$.   	\\ 
				Keep: $y_4=0$.  							\\
				
			\item \emph{Similarly for $v_5,v_6,v_7,v_8$}. \\

			\item $h(i) > 0$ for all $i \sim 9$. 	\\ 
			 Update: $y_9=1$.  
				\begin{table}[H] \centering  \begin{tabular}{c|cccccccccc}
			$	v_i	$	&	$v_1$	&	$v_2$	&	$v_3$	&	$v_4$	&	$v_5$	&	$v_6$	&	$v_7$	&	$v_8$	&	$v_9$	&	$v_{10}$	\\ \hline
				$	h(i)$	&	0	&	0& 	0 &	0	&	0	&	0	&	0	&	1	&0	&	0	
		\end{tabular}  					\end{table}						
			
					\item $h(10)=0$ and $10\sim 10$.   	\\ 
				Keep: $y_{10}=0$.  \\
		\item STOP: All vertices scanned. 		\\
					
\end{itemize}
	
 \item[\textbf{Stage 2:}] Scan neighbourhoods in reverse order.
				\begin{itemize}
					\item $h(10)=0$ and $T_{10} = \{9\} \subseteq T$. \\
							Update: $x_{10}=1$ and $T=\{1,2,3,4,5,6,7,8,10\}$. \\
					\item $h(9)=0$ but $T_9 = \{9\} \nsubseteq T$. \\
							Keep: $x_9=0$. 	\\
					\item $h(8)=1$.
							Keep: $x_8=0$. 	\\		
					\item $h(7)=0$ but $T_7 = \{9\} \nsubseteq T$. \\
							Keep: $x_7=0$. 	\\			
					\item $h(6)=0$ and $T_6 = \{3\} \subseteq T$. \\
							Update: $x_6=1$ and $T=\{1,2,4,5,6,7,8,10\}$. \\	
					\item $h(5)=0$ and $T_5 = \{2\} \subseteq T$. \\
							Update: $x_5=1$ and $T=\{1,4,5,6,7,8,10\}$. \\	
					\item $h(4)=0$ and $T_4 = \{1\} \subseteq T$. \\
							Update: $x_4=1$ and $T=\{4,5,6,7,8,10\}$. \\		
							
					\item STOP: $\sum_{v_i\in V} y_i =	\sum_{v_i\in V} x_i $											
				\end{itemize}
\end{enumerate}
	\noindent Therefore $\{	1,2,3,9\}$ is a maximum 2-packing and $\{4,5,6,10\}$ is a  minimum dominating set of $G$.
					
\end{example}

\begin{example}\label{ex:BDfarb}   Using Brewster and Duchesne's modification of Farber's algorithm (Algorithm \ref{alg:BDfarb}), we find a maximum multipacking and minimum dominating broadcast for the graph $G$ in Figure \ref{fig:exFarb}.  For space, we use the notation $(i,k)$ in place of $N_k[v_i]$. 
\begin{table}[H]
	\centering
	
\begin{tabular}{c|C{5mm}C{5mm}C{5mm}C{5mm}C{5mm}C{5mm}C{5mm}C{5mm}C{5mm}C{5mm}C{5mm}C{5mm}C{5mm}C{5mm}C{5mm}C{5mm}C{5mm}} 
																																																							
$	N_k[v_i]	$	&	\scriptsize	$(1,1)$	&	\scriptsize	$(2,1)$	&	\scriptsize	$(3,1)$	&	\scriptsize	$(4,1)$	&	\scriptsize	$(5,1)$	&	\scriptsize	$(6,1)$	&	\scriptsize	$(7,1)$	&	\scriptsize	$(8,1)$	&	\scriptsize	$(9,1)$	&	\scriptsize	$(7,2)$	&	\scriptsize	$(10,1)$	&	\scriptsize	$(8,2)$	&	\scriptsize	$(9,2)$	&	\scriptsize	$(10,2)$	&	\scriptsize	$(9,3)$	&	\scriptsize	$(10,3)$	&	\scriptsize	$(10,4)$	\\	\hline
$	v_1	$	&	1	&	0	&	0	&	1	&	0	&	0	&	0	&	0	&	0	&	1	&	0	&	0	&	0	&		0	&		1	&		0	&		1	\\	
$	v_2	$	&	0	&	1	&	0	&	0	&	1	&	0	&	0	&	0	&	0	&	1	&	0	&	0	&	0	&		0	&		1	&		0	&		1	\\	
$	v_3	$	&	0	&	0	&	1	&	0	&	0	&	1	&	0	&	0	&	0	&	0	&	0	&	1	&	0	&		0	&		0	&		1	&		1	\\	
$	v_4	$	&	1	&	0	&	0	&	1	&	0	&	0	&	1	&	0	&	0	&	1	&	0	&	0	&	1	&		0	&		1	&		1	&		1	\\	
$	v_5	$	&	0	&	1	&	0	&	0	&	1	&	0	&	1	&	0	&	0	&	1	&	0	&	0	&	1	&		0	&		1	&		1	&		1	\\	
$	v_6	$	&	0	&	0	&	1	&	0	&	0	&	1	&	0	&	1	&	0	&	0	&	0	&	1	&	0	&		1	&		1	&		1	&		1	\\	
$	v_7	$	&	0	&	0	&	0	&	1	&	1	&	0	&	1	&	0	&	1	&	1	&	0	&	0	&	1	&		1	&		1	&		1	&		1	\\	
$	v_8	$	&	0	&	0	&	0	&	0	&	0	&	1	&	0	&	1	&	0	&	0	&	1	&	1	&	1	&		1	&		1	&		1	&		1	\\	
$	v_9	$	&	0	&	0	&	0	&	0	&	0	&	0	&	1	&	0	&	1	&	1	&	1	&	1	&	1	&		1	&		1	&		1	&		1	\\	
$	v_{10}	$	&	0	&	0	&	0	&	0	&	0	&	0	&	0	&	1	&	1	&	1	&	1	&	1	&		1	&		1	&		1	&		1	&		1	\\				
\end{tabular}
 \caption{The extended neighbourhood matrix of $T$}
\end{table}

\begin{enumerate}[leftmargin=1in]

\item[\textbf{Initial:}] $y_i=0$ and $x_{i,k}=0$ for all $v_i\in V(G)$ \\ $T=\{1,2,3,4,5,6,7,8,9,10\}$.

\begin{table}[H]
	\centering
\begin{tabular}{c|C{5mm}C{5mm}C{5mm}C{5mm}C{5mm}C{5mm}C{5mm}C{5mm}C{5mm}C{5mm}C{5mm}C{5mm}C{5mm}C{5mm}C{5mm}C{5mm}C{5mm}} 
$	N_k[v_i]	$	&	\scriptsize	$(1,1)$	&	\scriptsize	$(2,1)$	&	\scriptsize	$(3,1)$	&	\scriptsize	$(4,1)$	&	\scriptsize	$(5,1)$	&	\scriptsize	$(6,1)$	&	\scriptsize	$(7,1)$	&	\scriptsize	$(8,1)$	&	\scriptsize	$(9,1)$	&	\scriptsize	$(7,2)$	&	\scriptsize	$(10,1)$	&	\scriptsize	$(8,2)$	&	\scriptsize	$(9,2)$	&	\scriptsize	$(10,2)$	&	\scriptsize	$(9,3)$	&	\scriptsize	$(10,3)$	&	\scriptsize	$(10,4)$	\\	\hline
$	h(i,k)	$	&	1	&	1	&	1	&	1	&	1	&	1	&	1	&	1	&		1	&		2	&		1	&		2	&		2	&		2	&		3	&		3	&		4			
\end{tabular} 
\end{table}

\item[\textbf{Stage 1:}] Scan rows in descending order.

\begin{itemize}

 \item $h(i,k) > 0$ for all $h(i,k)$ such that $ 1 \sim_m (i,k)$. 	\\
 			 Update: $y_1=1$.  
\begin{table}[H]
	\centering 			 
\begin{tabular}{c|C{5mm}C{5mm}C{5mm}C{5mm}C{5mm}C{5mm}C{5mm}C{5mm}C{5mm}C{5mm}C{5mm}C{5mm}C{5mm}C{5mm}C{5mm}C{5mm}C{5mm}} 
$	N_k[v_i]	$	&	\scriptsize	$(1,1)$	&	\scriptsize	$(2,1)$	&	\scriptsize	$(3,1)$	&	\scriptsize	$(4,1)$	&	\scriptsize	$(5,1)$	&	\scriptsize	$(6,1)$	&	\scriptsize	$(7,1)$	&	\scriptsize	$(8,1)$	&	\scriptsize	$(9,1)$	&	\scriptsize	$(7,2)$	&	\scriptsize	$(10,1)$	&	\scriptsize	$(8,2)$	&	\scriptsize	$(9,2)$	&	\scriptsize	$(10,2)$	&	\scriptsize	$(9,3)$	&	\scriptsize	$(10,3)$	&	\scriptsize	$(10,4)$	\\	\hline
$	h(i,k)	$	&		0	&		1	&		1	&		0	&		1	&		1	&		1	&		1	&		1	&		1	&		1	&		2	&		2	&		2	&		2	&		3	&		3			\end{tabular} 
\end{table}

 \item $h(i,k) > 0$ for all $h(i,k)$ such that $ 2 \sim_m (i,k)$.  	\\
 			 Update: $y_2=1$.  
\begin{table}[H]
	\centering
\begin{tabular}{c|C{5mm}C{5mm}C{5mm}C{5mm}C{5mm}C{5mm}C{5mm}C{5mm}C{5mm}C{5mm}C{5mm}C{5mm}C{5mm}C{5mm}C{5mm}C{5mm}C{5mm}} 
$	N_k[v_i]	$	&	\scriptsize	$(1,1)$	&	\scriptsize	$(2,1)$	&	\scriptsize	$(3,1)$	&	\scriptsize	$(4,1)$	&	\scriptsize	$(5,1)$	&	\scriptsize	$(6,1)$	&	\scriptsize	$(7,1)$	&	\scriptsize	$(8,1)$	&	\scriptsize	$(9,1)$	&	\scriptsize	$(7,2)$	&	\scriptsize	$(10,1)$	&	\scriptsize	$(8,2)$	&	\scriptsize	$(9,2)$	&	\scriptsize	$(10,2)$	&	\scriptsize	$(9,3)$	&	\scriptsize	$(10,3)$	&	\scriptsize	$(10,4)$	\\	\hline
$	h(i,k)	$	&		0	&		0	&		1	&		0	&		0	&		1	&		1	&		1	&		1	&		0	&		1	&		2	&		2	&		2	&		1	&		3	&		2				
\end{tabular} 
\end{table}

 \item $h(i,k) > 0$ for all $h(i,k)$ such that $ 2 \sim_m (i,k)$.  	\\
  			 Update: $y_3=1$.  
\begin{table}[H]
	\centering
\begin{tabular}{c|C{5mm}C{5mm}C{5mm}C{5mm}C{5mm}C{5mm}C{5mm}C{5mm}C{5mm}C{5mm}C{5mm}C{5mm}C{5mm}C{5mm}C{5mm}C{5mm}C{5mm}} 
$	N_k[v_i]	$	&	\scriptsize	$(1,1)$	&	\scriptsize	$(2,1)$	&	\scriptsize	$(3,1)$	&	\scriptsize	$(4,1)$	&	\scriptsize	$(5,1)$	&	\scriptsize	$(6,1)$	&	\scriptsize	$(7,1)$	&	\scriptsize	$(8,1)$	&	\scriptsize	$(9,1)$	&	\scriptsize	$(7,2)$	&	\scriptsize	$(10,1)$	&	\scriptsize	$(8,2)$	&	\scriptsize	$(9,2)$	&	\scriptsize	$(10,2)$	&	\scriptsize	$(9,3)$	&	\scriptsize	$(10,3)$	&	\scriptsize	$(10,4)$	\\	\hline
$	h(i,k)	$	&		0	&		0	&		0	&		0	&		0	&		0	&		1	&		1	&		1	&		0	&		1	&		1	&		2	&		2	&		1	&		2	&		1		
\end{tabular} 
\end{table}

 \item $h(1,1) = 0$,	\\
 			 Keep: $y_4=0$.  \\

 \item \emph{Similarly for }$5,6,7,8,9,10$.

		\item STOP: All vertices scanned. 		\\
					
\end{itemize}
	
 \item[\textbf{Stage 2:}] Scan neighbourhoods in reverse order.
				\begin{itemize}
					\item $h(10,4)>0$.  Keep $x_{10,4}=0$. \\
					\item \emph{Similarly for }$(10,3),(9,3),(10,2),(9,2),(8,2),(10,1)$. \\
					
					\item $h(7,2)=0$ and $T_{7,2} = \{1,2\} \subseteq T$. \\
							Update: $x_{7,2}=1$ and $T=\{3,4,5,6,7,8,9,10\}$. \\		
					
					\item $h(9,1)>0$.  Keep $x_{9,1}=0$. \\
					\item \emph{Similarly for }$(8,1),(7,1)$. \\
		
					\item $h(6,1)=0$ and $T_{6,1} = \{3\} \subseteq T$. \\
							Update: $x_{6,1}=1$ and $T=\{4,5,6,7,8,9,10\}$. \\		
									
					\item STOP: $\sum_{v_i\in V} y_i = 	\sum_{v_i\in V, k} x_{i,k} $																																
									
				\end{itemize}
\end{enumerate}
	\noindent Therefore $\{	v_1,v_2, v_3\}$ is a maximum 2-packing and the broadcast 
	\begin{align*}
		f(v_i) = \begin{cases}
		2 & \text{ for } i=7\\
		1 & \text{ for } i=6 \\
		0 & \text{ otherwise,}\\
		\end{cases}
	\end{align*}
	 is a  minimum dominating broadcast of $G$.

\end{example}


\subsection{The Tree Multipacking Algorithm} \label{ap:mpTree}

In this section, we provide the original algorithm from \cite{LThesis} for finding a maximum multipacking of a tree.  Before presenting said algorithm, we first supply the necessary definitions and notations.

Let $P:v_{0},...,v_{d}$ be a diametrical
path of a tree $T$ with $\operatorname{diam}(T)=d$. For each $v_{i}\in V(P)$,
let $U_{i}$ be the set of all vertices of $T$ that are connected to $v_{i}$ by
a (possibly trivial) path internally disjoint from $P$. Let $u_{i}$ be a
vertex in $U_{i}$ at maximum distance from $v_{i}$, and let $B_{i}$ be the
$v_{i}-u_{i}$ path. The \emph{shadow tree} $S_{T,P}$ \emph{of }$T$ \emph{with
respect to} $P$ is the subtree of $T$ induced by $\bigcup_{i=0}^{d}V(B_{i})$.
If $T\cong S_{T,P}$ for some diametrical path $P$ of $T$, then $T$ is also
called a \emph{shadow tree}. Note that a shadow tree has maximum degree at
most three.

Consider a shadow tree $S_{T,P}$. If $U_{i}-\{v_{i}\}\neq\varnothing$, we call
$v_{i}$ a \emph{branch vertex} and the $v_{i}-u_{i}$ path $B_{i}$ a
\emph{branch}. Furthermore, for  $\alpha_{i}=d(v_{i},u_{i})\geq1$, the tree
$\Delta_{i}$ induced by $\left\{  v_{i-\alpha_{i}},\dots,v_{i-1}\right\}  \cup
V(B_{i})\cup\left\{  v_{i+1},\dots,v_{i+\alpha_{i}}\right\}  $ is called the
\emph{triangle at} $i$. If the vertex subset $\{v_{i-\alpha_{i}},\dots
,v_{i},\dots,v_{i+\alpha_{i}}\}$ of the triangle $\Delta_{i}$ is contained in
the vertex subset $\{v_{j-\alpha_{j}},\dots,v_{j},\dots,v_{j+\alpha_{j}}\}$ of
the triangle $\Delta_{j}$, then $\Delta_{i}$ is called a \emph{nested
triangle}. 
A \emph{free edge} is an edge of $S_{T,P}$ that is not in any
triangle; note that all free edges of $S_{T,P}$ lie on $P$.

The triangles of  $S_{T,P}$ are labeled in order of their occurrence on $P$ and are denoted  $\Delta _{i_1},\Delta_{i_2},\dots,\Delta_{i_c}$.  For simplicity, we abuse notation and denote $\Delta _{i_1}$ as $\Delta _{1}$, and $\Delta _{i_c}$ as  $\Delta _{c}$. A free edge on $P$ that comes before
$\Delta_{1}$ is called a \emph{leading free edge}; likewise, a free edge that
comes after $\Delta_{c}$ is called a \emph{trailing free edge}. If $e$ is a
free edge of $S_{T,P}$, we also call $e$ a \emph{free edge of }$T$\emph{\ with
respect to} $P$.  A set $M$ of edges of the diametrical path $P$ of the tree $T$ is a
\emph{split-}$P$ \emph{set} if each component $T^{\prime}$ of $T-M$ has a
positive even diameter and $P^{\prime}=T^{\prime}\cap P$ is a diametrical path
of $T^{\prime}$. A \emph{split-set} of $T$ is a split-$P$ set for some
diametrical path $P$ of $T$. An edge in any split-set of $T$ is a
\emph{split-edge}.  
The requirement that $P^{\prime}=T^{\prime}\cap P$ be a
diametrical path of $T^{\prime}$ implies that each split-edge is a free edge.
However, not all free edges are split-edges.

\begin{figure}[H]
\centering
\begin{tikzpicture}					
					\foreach \i/\x in {0/0, 1/0.5,2/1,3/1.5, 4/2, 5/2.5, 6/3, 7/3.5, 8/4, 9/4.5, 10/5, 11/5.5, 12/6, 13/6.5}
							\node [std](x\i) at (\x,1) {};		
					
					\foreach \f/\t in {0/1,1/2,2/3,3/4,4/5,5/6,6/7,7/8,8/9,9/10,10/11,11/12,12/13}
						\draw (x\f) -- (x\t);
										
						\foreach \i/\y in	{0/1, 1/1.5,2/2}
							\node [std](y\i) at (1,\y) {};	
							
						\foreach \f/\t in {0/1,1/2}
							\draw (y\f) -- (y\t);
						
						\draw[tr] (x0)--(y2)--(x4);		
					
					\foreach \i/\c in {0/1, 1/1.5,2/2,3/2.5, 4/3}
							\node [std](c\i) at (4,\c) {};		
					
					\foreach \f/\t in {0/1,1/2,2/3,3/4}
						\draw (c\f) -- (c\t);					
					
					\draw[tr] (x4)--(c4)--(x12);		
					
				\node[below=.2cm of x0][color=black](v0){$v_0$};		
				\node[below=.2cm of x2][color=black](v2){$v_2$};	
				\node[below=.2cm of x8][color=black](vc){$v_c$};	
				\node[below=.2cm of x13][color=black](vd){$v_d$};	
				\node[right=.2cm of c1][color=black](uc1){$u_{c,1}$};	
				\node[right=.2cm of c4][color=black](uca){$u_{c,\alpha}$};	
				
				\node[draw=ProcessBlue, thick, rounded corners, dashed, fit=(x8) (x13)] (Qc) {};
				\node[right=.2cm of x13][color=ProcessBlue](Q){$Q_c$};	
				
				\node[draw=Plum, thick, rounded corners, dashed, fit=(x8) (c4)] (Bc) {};
				\node[above=.2cm of c4][color=Plum](B){$B_c$};		
				
				\node (hend) at (7.3,1.05) {};	
				\node (hstart) at (0,0.95) {};	
				\node[draw= BurntOrange, thick, rounded corners, dashed, fit=(hend) (hstart)] (Pb) {};
				\node[left=.2cm of x0][color=BurntOrange](P){$P$};		
								
				\draw[color=Green, line width = 2pt] (x4)--(x5);
				\node[below right =.2cm and .3cm of x3][color=Green](ef){$e_f$};	
				
				\draw[color=Green, line width = 2pt] (x11)--(x12);
				\node[below right =.2cm and .3cm of x10][color=Green](el){$e_l$};				
										
		\end{tikzpicture}
\caption{A labelled shadow tree $S_{T,P}$.}
\label{fig:definitions}
\end{figure}
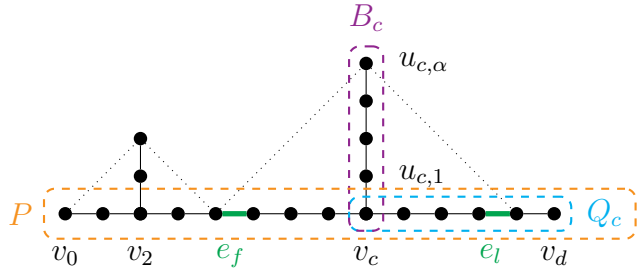

We illustrate the following notation in Figure \ref{fig:definitions}. Let $c$
be the highest index such that $v_{c}$ is a branch vertex of $T$. The
subpath $Q_{c}:v_{c},...,v_{d}$ of $P$ is called the \emph{trailing endpath}
of $T$. The branch of $T$ that starts at $v_{c}$ is the path $%
B_{c}:v_{c}=u_{c,0},u_{c,1},...,u_{c,\alpha}$ of length $\alpha$, and is
called the \emph{last branch} of $T$. The triangle $\Delta_{c}$ associated
with $B_{c}$ is called the \emph{last triangle} of $T$. For brevity we also
write $B_{c}$ and $Q_{c}$ for $V(B_{c})$ and $V(Q_{c})$, respectively. We
denote the lengths of $B_{c}$ and $Q_{c}$ by $\ell(B_{c})$ and $\ell(Q_{c}) $%
, respectively; note that $\alpha=\ell(B_{c})\leq\ell(Q_{c})=d-c $. The
first and last edges of $\Delta_{c}$ on $P$ are $e_{f}=v_{c-\alpha}v_{c-%
\alpha+1}$ and $e_{\ell }=v_{c+\alpha-1}v_{c+\alpha}$, respectively.

\setcounter{algocf}{3}
\begin{algorithm}[H]
\KwIn{Shadow tree $T$ with no nested triangles, diametrical path $P=\left\{v_0, v_1, \dots v_d\right\}$}
\KwOut{A maximum multipacking $M$ of $T$}

$M \gets \varnothing$ \\
$S \gets \varnothing$ \\
\While{$T \neq P_1, P_2, P_3$} {

	\If{$c=i_1$ and $l(B_c) > l(Q_1)$}{
		$P \gets P - Q_1 + B_c$ \\
		\textbf{comment: }{$Q_1$ becomes $B_c$}
	}
	\If{$\Delta_{i_j}$ is a nested triangle}{
		$T \gets T- (B_{i_j} - \left\{v_{i_j}\right\})$ 
	}
  \uIf{number of trailing free edges $\geq 3$} {
    $M \gets M\cup \left\{v_d\right\}$ \\
    $P \gets P - \left\{v_d, v_{d-1}, v_{d-2}\right\}$ \\
    $T \gets T - \left\{v_d, v_{d-1}, v_{d-2}\right\}$ \\
  }
  \uElseIf{number of trailing free edges $= 2$} {
    $M \gets M\cup \left\{v_d\right\}$ \\
    $P \gets P - \left\{v_d, v_{d-1}, v_{d-2}\right\}$ \\
    $T \gets T - \left\{v_d, v_{d-1}, v_{d-2}\right\}$ \\
    $P \gets P - Q_c + B_c$ \\ 
    		\textbf{comment: }{$Q_c$ becomes $B_c$}
  }
  \uElseIf{number of trailing free edges $= 1$}{
  	$S \gets S \cup \left\{(v_c, v_{c-1})\right\}$\\
    $T \gets T \cup \left\{u_{c-1,1}, u_{c-1,2}, \dots, u_{c-1, \alpha}\right\}$\\
    $T \gets T - (B_c - \left\{v_c\right\})$
  }
  \Else{
    $T \gets T - \left\{u_{c,\alpha}\right\}$ 
  }
}
	$M \gets M\cup \left\{v_0\right\}$\;
	
	\ForAll{$(u,v) \in S$}{
		\If{$u \in M$}{
			$M \gets (M - \left\{u\right\}) \cup \left\{v\right\}$		
		}

	}
	
\Return{$M$}
\caption{{\sc FindTreeMP} finds a maximum multipacking of a tree}
\label{algo}
\end{algorithm}

We conclude this section by presenting an example of Algorithm \ref{algo} in use in Figure \ref{fig:algo2}.

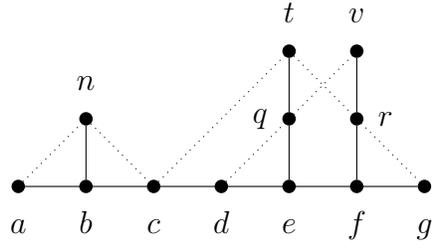
\begin{figure}[H]
			\centering
				\begin{tikzpicture}[scale=0.9]
							\node [std](a) at (0,0)  {};
								\node [below=2mm of a](La){$a$};
							\node [std](b) at (1,0)  {};	
								\node [below=1mm of b](Lb){$b$};
							\node [std](c) at (2,0) {};	
								\node [below=2mm of c](Lc){$c$};
							\node [std](d) at (3,0)  {};	
								\node [below=1mm of d](Ld){$d$};
							\node [std](e) at (4,0)  {};	
								\node [below=2mm of e](Le){$e$};
							\node [std](f) at (5,0)  {};	
								\node [below=1mm of f](Lf){$f$};
							\node [std](g) at (6,0) {};	
								\node [below=2mm of g](Lg){$g$};
							\node [std](h) at (7,0)  {};
								\node [below=1.5mm of h](Lh){$h$};	
							\node [std](i) at (8,0)  {};	
								\node [below=1.5mm of i](Li){$i$};
							\node [std](j) at (9,0) {};						
								\node [below=1.5mm of j](Lj){$j$};
							\node [std](k) at (10,0) {};						
								\node [below=1.5mm of k](Lk){$k$};
							\node [std](l) at (11,0) {};						
								\node [below=1.5mm of l](Ll){$l$};
							\node [std](m) at (12,0) {};						
								\node [below=2mm of m](Lm){$m$};																								
								
						\draw (a)--(m);		
						
							\node [std](n) at (0.5,1)  {};	
								\node [above=1.5mm of n](Ln){$n$};
							\node [std](o) at (1.5,1) {};						
								\node [above=1.5mm of o](Lo){$o$};
							\node [std](p) at (3,1) {};						
								\node [left=0.5mm of p](Lp){$p$};
							\node [std](q) at (4,1) {};						
								\node [left=0.5mm of q](Lq){$q$};
							\node [std](r) at (6,1) {};						
								\node [right=0.5mm of r](Lr){$r$};					
							\node [std](s) at (7,1) {};						
								\node [right=1mm of s](Ls){$s$};						
							
							\draw (b)--(n);
							\draw (b)--(o);
							\draw (d)--(p);
							\draw (e)--(q);
							\draw (g)--(r);
							\draw (g)--(s);				
							
							\node [std](t) at (3.5,2) {};						
								\node [above=1.5mm of t](Lt){$t$};
							\node [std](u) at (4.5,2) {};						
								\node [above=1.5mm of u](Lu){$u$};
							\node [std](v) at (6,2) {};						
								\node [above=1.5mm  of v](Lq){$v$};
							\node [std](w) at (7,2) {};						
								\node [above=1.5mm  of w](Lw){$w$};					
							\node [std](x) at (8,2) {};						
								\node [above=1.5mm  of x](Lx){$x$};	
								
							\draw(q)--(t);
							\draw(u)--(q);
							\draw(r)--(v);
							\draw(s)--(w);
							\draw(s)--(x);	
							
				\node[below=1cm of g][color=black](BLa){$(1)$ A tree $T'$ with diametrical path $P = \left\{a,b,c,\dots, m\right\} $};				
				\end{tikzpicture}
				\vspace{0.5cm}

				\begin{tikzpicture}[scale=0.9]
							\node [std](a) at (0,0)  {};
								\node [below=2mm of a](La){$a$};
							\node [std](b) at (1,0)  {};	
								\node [below=1mm of b](Lb){$b$};
							\node [std](c) at (2,0) {};	
								\node [below=2mm of c](Lc){$c$};
							\node [std](d) at (3,0)  {};	
								\node [below=1mm of d](Ld){$d$};
							\node [std](e) at (4,0)  {};	
								\node [below=2mm of e](Le){$e$};
							\node [std](f) at (5,0)  {};	
								\node [below=1mm of f](Lf){$f$};
							\node [std](g) at (6,0) {};	
								\node [below=2mm of g](Lg){$g$};
							\node [std](h) at (7,0)  {};
								\node [below=1.5mm of h](Lh){$h$};	
							\node [std](i) at (8,0)  {};	
								\node [below=1.5mm of i](Li){$i$};
							\node [std](j) at (9,0) {};						
								\node [below=1.5mm of j](Lj){$j$};
							\node [std](k) at (10,0) {};						
								\node [below=1.5mm of k](Lk){$k$};
							\node [std](l) at (11,0) {};						
								\node [below=1.5mm of l](Ll){$l$};
							\node [std](m) at (12,0) {};						
								\node [below=2mm of m](Lm){$m$};																								
								
						\draw (a)--(m);		
						
							\node [std](n) at (1,1)  {};	
								\node [above=1.5mm of n](Ln){$n$};
							\node [std](q) at (4,1) {};						
								\node [left=0.5mm of q](Lq){$q$};
							\node [std](r) at (6,1) {};						
								\node [right=0.5mm of r](Lr){$r$};	
							
							\draw (b)--(n);
							\draw (e)--(q);
							\draw (g)--(r);
							
							\node [std](t) at (4,2) {};						
								\node [above=1.5mm of t](Lt){$t$};
							\node [std](v) at (6,2) {};						
								\node [above=1.5mm  of v](Lq){$v$};
								
							\draw(q)--(t);
							\draw(r)--(v);
														
				\draw[tr] (a)--(n)--(c);
				\draw[tr] (c)--(t)--(g);
				\draw[tr] (e)--(v)--(i);
				
				\node [mp](m) at (12,0) {};		
							
				\node[below=1cm of g][color=black](BLa){$(2)$ Create shadow  tree $T = S_{T',P}$ of $T'$ with no nested triangles.  Add $m$ to $M$.};				
				\end{tikzpicture}						
				
			\vspace{1cm}	
				\begin{tikzpicture}[scale=0.9]
							\node [std](a) at (0,0)  {};
								\node [below=2mm of a](La){$a$};
							\node [std](b) at (1,0)  {};	
								\node [below=1mm of b](Lb){$b$};
							\node [std](c) at (2,0) {};	
								\node [below=2mm of c](Lc){$c$};
							\node [std](d) at (3,0)  {};	
								\node [below=1mm of d](Ld){$d$};
							\node [std](e) at (4,0)  {};	
								\node [below=2mm of e](Le){$e$};
							\node [std](f) at (5,0)  {};	
								\node [below=1mm of f](Lf){$f$};
							\node [std](g) at (6,0) {};	
								\node [below=2mm of g](Lg){$g$};
							\node [std](h) at (7,0)  {};
								\node [below=1.5mm of h](Lh){$h$};	
							\node [std](i) at (8,0)  {};	
								\node [below=1.5mm of i](Li){$i$};
							\node [std](j) at (9,0) {};						
								\node [below=1.5mm of j](Lj){$j$};														
								
						\draw (a)--(j);		
						
							\node [std](n) at (1,1)  {};	
								\node [above=1.5mm of n](Ln){$n$};
							\node [std](q) at (4,1) {};						
								\node [left=0.5mm of q](Lq){$q$};
							\node [std](r) at (6,1) {};						
								\node [right=0.5mm of r](Lr){$r$};	
							
							\draw (b)--(n);
							\draw (e)--(q);
							\draw (g)--(r);
							
							\node [std](t) at (4,2) {};						
								\node [above=1.5mm of t](Lt){$t$};
							\node [std](v) at (6,2) {};						
								\node [above=1.5mm  of v](Lq){$v$};
								
							\draw(q)--(t);
							\draw(r)--(v);
														
				\draw[tr] (a)--(n)--(c);
				\draw[tr] (c)--(t)--(g);
				\draw[tr] (e)--(v)--(i);
				
				\node[below=1cm of e][color=black](BLa){$(3)$ Delete $\left\{k,l,m  \right\}$ from $T$.};	
				\end{tikzpicture}				
				
			\vspace{1cm}	
				\begin{tikzpicture}[scale=0.9]
							\node [std](a) at (0,0)  {};
								\node [below=2mm of a](La){$a$};
							\node [std](b) at (1,0)  {};	
								\node [below=1mm of b](Lb){$b$};
							\node [std](c) at (2,0) {};	
								\node [below=2mm of c](Lc){$c$};
							\node [std](d) at (3,0)  {};	
								\node [below=1mm of d](Ld){$d$};
							\node [std](e) at (4,0)  {};	
								\node [below=2mm of e](Le){$e$};
							\node [std](f) at (5,0)  {};	
								\node [below=1mm of f](Lf){$f$};
							\node [std](g) at (6,0) {};	
								\node [below=2mm of g](Lg){$g$};
							\node [std](h) at (7,0)  {};
								\node [below=1.5mm of h](Lh){$h$};	
							\node [std](i) at (8,0)  {};	
								\node [below=1.5mm of i](Li){$i$};
							\node [std](j) at (9,0) {};						
								\node [below=1.5mm of j](Lj){$j$};
								
						\draw (a)--(j);		
						
							\node [std](n) at (1,1)  {};	
								\node [above=1.5mm of n](Ln){$n$};
							\node [std](q) at (4,1) {};						
								\node [left=0.5mm of q](Lq){$q$};
							\node [std](r) at (5,1) {};						
								\node [right=0.5mm of r](Lr){$r$};		
							
							\draw (b)--(n);
							\draw (e)--(q);
							\draw (f)--(r);
							
							\node [std](t) at (4,2) {};						
								\node [above=1.5mm of t](Lt){$t$};
							\node [std](v) at (5,2) {};						
								\node [above=1.5mm  of v](Lq){$v$};
								
							\draw(q)--(t);
							\draw(r)--(v);
														
				\draw[tr] (a)--(n)--(c);
				\draw[tr] (c)--(t)--(g);
				\draw[tr] (d)--(v)--(h);
				
							\node [mp](j) at (9,0) {};			
														
				\node[below=1cm of e][color=black](BLa){$(4)$ Shift $B_c$ to $v_{c-1}$.  Add $(g,f)$ to $S$.  Add $j$ to $M$.};				
				\end{tikzpicture}			
\label{fig:algo}
\end{figure}								

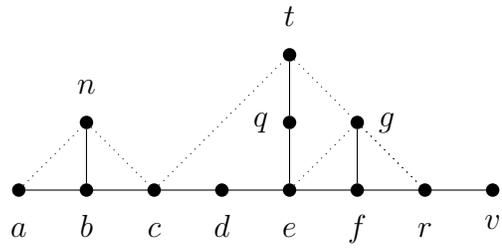
\begin{figure}
			\centering
				\begin{tikzpicture}[scale=0.9]
							\node [std](a) at (0,0)  {};
								\node [below=2mm of a](La){$a$};
							\node [std](b) at (1,0)  {};	
								\node [below=1mm of b](Lb){$b$};
							\node [std](c) at (2,0) {};	
								\node [below=2mm of c](Lc){$c$};
							\node [std](d) at (3,0)  {};	
								\node [below=1mm of d](Ld){$d$};
							\node [std](e) at (4,0)  {};	
								\node [below=2mm of e](Le){$e$};
							\node [std](f) at (5,0)  {};	
								\node [below=1mm of f](Lf){$f$};
							\node [std](g) at (6,0) {};	
								\node [below=2mm of g](Lg){$g$};
								
						\draw (a)--(g);		
						
							\node [std](n) at (1,1)  {};	
								\node [above=1.5mm of n](Ln){$n$};
							\node [std](q) at (4,1) {};						
								\node [left=0.5mm of q](Lq){$q$};
							\node [std](r) at (5,1) {};						
								\node [right=0.5mm of r](Lr){$r$};					
				
							\draw (b)--(n);
							\draw (e)--(q);
							\draw (f)--(r);
							
							\node [std](t) at (4,2) {};						
								\node [above=1.5mm of t](Lt){$t$};
							\node [std](v) at (5,2) {};						
								\node [above=1.5mm  of v](Lq){$v$};
								
							\draw(q)--(t);
							\draw(r)--(v);
														
				\draw[tr] (a)--(n)--(c);
				\draw[tr] (c)--(t)--(g);
				\draw[tr] (d)--(v);
				
				\node[below=1cm of d][color=black](BLa){$(5)$ Delete $\left\{h,i,j  \right\}$ from $T$.};					\end{tikzpicture}	
		
		\vspace{1cm}		
				\begin{tikzpicture}[scale=0.9]
							\node [std](a) at (0,0)  {};
								\node [below=2mm of a](La){$a$};
							\node [std](b) at (1,0)  {};	
								\node [below=1mm of b](Lb){$b$};
							\node [std](c) at (2,0) {};	
								\node [below=2mm of c](Lc){$c$};
							\node [std](d) at (3,0)  {};	
								\node [below=1mm of d](Ld){$d$};
							\node [std](e) at (4,0)  {};	
								\node [below=2mm of e](Le){$e$};
							\node [std](f) at (5,0)  {};	
								\node [below=1mm of f](Lf){$f$};
							\node [std](r) at (6,0) {};	
								\node [below=2mm of r](Lr){$r$};
							\node [std](v) at (7,0)  {};	
								\node [below=1mm of v](Lv){$v$};			
								
						\draw (a)--(v);		
						
							\node [std](n) at (1,1)  {};	
								\node [above=1.5mm of n](Ln){$n$};
							\node [std](q) at (4,1) {};						
								\node [left=0.5mm of q](Lq){$q$};
							\node [std](g) at (5,1) {};						
								\node [right=0.5mm of g](Lg){$g$};
							
							\draw (b)--(n);
							\draw (e)--(q);
							\draw (f)--(g);
							
							\node [std](t) at (4,2) {};						
								\node [above=1.5mm of t](Lt){$t$};
								
							\draw(q)--(t);
														
				\draw[tr] (a)--(n)--(c);
				\draw[tr] (c)--(t)--(r);
				\draw[tr] (e)--(g)--(r);
				
				\node[below=1cm of e][color=black](BLa){$(6)$ Swap $Q_c$ and $B_c$.  $P$ becomes $P -B_c \cup Q_c$.};
		\end{tikzpicture}			

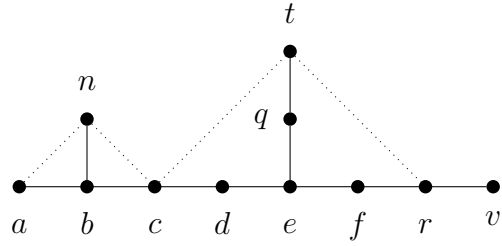
\vspace{1cm}
				\begin{tikzpicture}[scale=0.9]
							\node [std](a) at (0,0)  {};
								\node [below=2mm of a](La){$a$};
							\node [std](b) at (1,0)  {};	
								\node [below=1mm of b](Lb){$b$};
							\node [std](c) at (2,0) {};	
								\node [below=2mm of c](Lc){$c$};
							\node [std](d) at (3,0)  {};	
								\node [below=1mm of d](Ld){$d$};
							\node [std](e) at (4,0)  {};	
								\node [below=2mm of e](Le){$e$};
							\node [std](f) at (5,0)  {};	
								\node [below=1mm of f](Lf){$f$};
							\node [std](r) at (6,0) {};	
								\node [below=2mm of r](Lr){$r$};
							\node [std](v) at (7,0)  {};	
								\node [below=1mm of v](Lv){$v$};			
								
						\draw (a)--(v);		
						
							\node [std](n) at (1,1)  {};	
								\node [above=1.5mm of n](Ln){$n$};
							\node [std](q) at (4,1) {};						
								\node [left=0.5mm of q](Lq){$q$};
							
							\draw (b)--(n);
							\draw (e) -- (q);
							
							\node [std](t) at (4,2) {};						
								\node [above=1.5mm of t](Lt){$t$};
								
							\draw(q)--(t);
														
				\draw[tr] (a)--(n)--(c);
				\draw[tr] (c)--(t)--(r);
				
				\node[below=1cm of e][color=black](BLa){$(7)$ Delete the nested triangle $\Delta_f$.};
		\end{tikzpicture}			
		
		
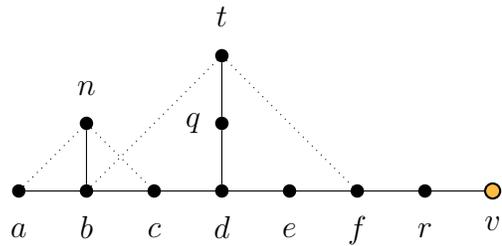
\vspace{1cm}		
				\begin{tikzpicture}			[scale=0.9]		
							\node [std](a) at (0,0)  {};
								\node [below=2mm of a](La){$a$};
							\node [std](b) at (1,0)  {};	
								\node [below=1mm of b](Lb){$b$};
							\node [std](c) at (2,0) {};	
								\node [below=2mm of c](Lc){$c$};
							\node [std](d) at (3,0)  {};	
								\node [below=1mm of d](Ld){$d$};
							\node [std](e) at (4,0)  {};	
								\node [below=2mm of e](Le){$e$};
							\node [std](f) at (5,0)  {};	
								\node [below=1mm of f](Lf){$f$};
							\node [std](r) at (6,0) {};	
								\node [below=2mm of r](Lr){$r$};
							\node [std](v) at (7,0)  {};	
								\node [below=1mm of v](Lv){$v$};			
								
						\draw (a)--(v);		
						
							\node [std](n) at (1,1)  {};	
								\node [above=1.5mm of n](Ln){$n$};
							\node [std](q) at (3,1) {};						
								\node [left=0.5mm of q](Lq){$q$};
							
							\draw (b)--(n);
							\draw (d) -- (q);
							
							\node [std](t) at (3,2) {};						
								\node [above=1.5mm of t](Lt){$t$};
								
							\draw(q)--(t);
														
				
				\draw[tr] (a)--(n)--(c);
				\draw[tr] (b)--(t)--(f);
				
				\node [mp](v) at (7,0) {};			
																	
				\node[below=1cm of e][color=black](BLa){$(8)$ Shift $B_c$ to $v_{c-1}$.  Add $(e,d)$ to $S$.  Add $v$ to $M$.};			
		\end{tikzpicture}	

\label{fig:algo2}
\end{figure}

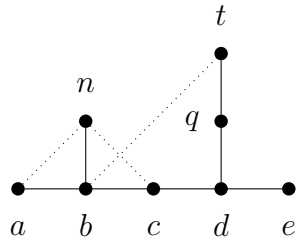
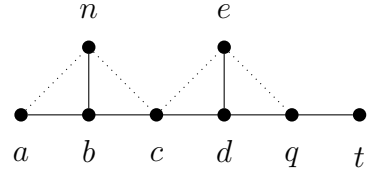
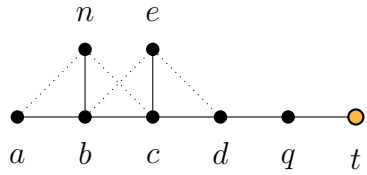
\begin{figure}
			\centering
				\begin{tikzpicture}[scale=0.9]					
							\node [std](a) at (0,0)  {};
								\node [below=2mm of a](La){$a$};
							\node [std](b) at (1,0)  {};	
								\node [below=1mm of b](Lb){$b$};
							\node [std](c) at (2,0) {};	
								\node [below=2mm of c](Lc){$c$};
							\node [std](d) at (3,0)  {};	
								\node [below=1mm of d](Ld){$d$};
							\node [std](e) at (4,0)  {};	
								\node [below=2mm of e](Le){$e$};	
								
						\draw (a)--(e);		
						
							\node [std](n) at (1,1)  {};	
								\node [above=1.5mm of n](Ln){$n$};
							\node [std](q) at (3,1) {};						
								\node [left=0.5mm of q](Lq){$q$};
							
							\draw (b)--(n);
							\draw (d) -- (q);
							
							\node [std](t) at (3,2) {};						
								\node [above=1.5mm of t](Lt){$t$};
								
							\draw(q)--(t);
														
				
				\draw[tr] (a)--(n)--(c);
				\draw[tr] (b)--(t);
				
				\node[below=1cm of c][color=black](BLa) {$(9)$ Delete $\left\{f,r,v  \right\}$ from $T$.};
		\end{tikzpicture}		
		\hspace{2cm}
				\begin{tikzpicture}[scale=0.9]
							\node [std](a) at (0,0)  {};
								\node [below=2mm of a](La){$a$};
							\node [std](b) at (1,0)  {};	
								\node [below=1mm of b](Lb){$b$};
							\node [std](c) at (2,0) {};	
								\node [below=2mm of c](Lc){$c$};
							\node [std](d) at (3,0)  {};	
								\node [below=1mm of d](Ld){$d$};
							\node [std](q) at (4,0)  {};	
								\node [below=2mm of q](Lq){$q$};	
							\node [std](t) at (5,0)  {};	
								\node [below=2mm of t](Lt){$t$};									
								
						\draw (a)--(t);		
						
							\node [std](n) at (1,1)  {};	
								\node [above=1.5mm of n](Ln){$n$};
							\node [std](e) at (3,1)  {};		
							\node [above=1.5mm of e](Le){$e$};
							
							\draw (b)--(n);
							\draw (d) -- (e);
							
														
				\draw[tr] (a)--(n)--(c);
				\draw[tr] (c)--(e)--(q);
				

	\node[below=1cm of d][color=black](BLa){$(10)$ Swap $Q_c$ and $B_c$.};
		\end{tikzpicture}			
		
		\vspace{1cm}
				\begin{tikzpicture}[scale=0.9]					
							\node [std](a) at (0,0)  {};
								\node [below=2mm of a](La){$a$};
							\node [std](b) at (1,0)  {};	
								\node [below=1mm of b](Lb){$b$};
							\node [std](c) at (2,0) {};	
								\node [below=2mm of c](Lc){$c$};
							\node [std](d) at (3,0)  {};	
								\node [below=1mm of d](Ld){$d$};
							\node [std](q) at (4,0)  {};	
								\node [below=2mm of q](Lq){$q$};	
							\node [std](t) at (5,0)  {};	
								\node [below=2mm of t](Lt){$t$};									
								
						\draw (a)--(t);		
						
							\node [std](n) at (1,1)  {};	
								\node [above=1.5mm of n](Ln){$n$};
							\node [std](e) at (2,1) {};						
								\node [above=1.5mm of e](Le){$e$};
							
							\draw (b)--(n);
							\draw (c) -- (e);
							
														
				\draw[tr] (a)--(n)--(c);
				\draw[tr] (b)--(e)--(d);
				
				\node [mp](t) at (5,0) {};		

	\node[below=1cm of c][color=black](BLa){$(11)$ Shift $B_c$.  Add $(d,c)$ to $S$.  Add $t$ to $M$.};			
		\end{tikzpicture}			
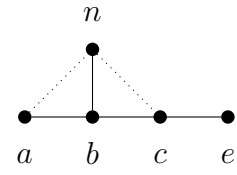
\hspace{1cm}
				\begin{tikzpicture}[scale=0.9]					
							\node [std](a) at (0,0)  {};
								\node [below=2mm of a](La){$a$};
							\node [std](b) at (1,0)  {};	
								\node [below=1mm of b](Lb){$b$};
							\node [std](c) at (2,0) {};	
								\node [below=2mm of c](Lc){$c$};
							\node [std](e) at (3,0)  {};	
							
								\node [below=2mm of e](Le){$e$};
						\draw (a)--(e);		
						
							\node [std](n) at (1,1)  {};	
								\node [above=1.5mm of n](Ln){$n$};
						
							\draw (b)--(n);
						
														
				\draw[tr] (a)--(n)--(c);
				

	\node[below=1cm of c][color=black](BLa){$(12)$ Delete $\left\{d,q,t  \right\}$ from $T$.};
		\end{tikzpicture}		

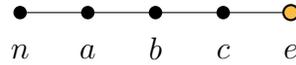
\vspace{0.5cm}
				\begin{tikzpicture}[scale=0.9]					
							\node [std](n) at (0,0)  {};
								\node [below=2mm of n](Ln){$n$};					
							\node [std](a) at (1,0)  {};
								\node [below=2mm of a](La){$a$};
							\node [std](b) at (2,0)  {};	
								\node [below=1mm of b](Lb){$b$};
							\node [std](c) at (3,0) {};	
								\node [below=2mm of c](Lc){$c$};
							\node [std](e) at (4,0)  {};								
								\node [below=2mm of e](Le){$e$};
								
						\draw (n)--(e);		
						
						
														
				
				\node [mp](e) at (4,0)  {};		
				
	\node[below=1cm of b][color=black](BLa){$(13)$ Shift $B_c$.  Add $(b,a)$ to $S$.  Add $e$ to $M$.};
		\end{tikzpicture}		
		
		
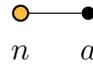
\vspace{0.5cm}
				\begin{tikzpicture}[scale=0.9]					
							\node [std](n) at (0,0)  {};
								\node [below=2mm of n](Ln){$n$};					
							\node [std](a) at (1,0)  {};
								\node [below=2mm of a](La){$a$};
					
						\draw (n)--(a);		
						
				\node [mp](n) at (0,0)  {};		
								
	\node[below=1cm of a][color=black](BLa){$(14)$ Delete $\left\{b,c,e  \right\}$ from $T$.  Add $n$ to $M$.};
		\end{tikzpicture}	

\vspace{1cm}
				\begin{tikzpicture}[scale=0.9]					
			\node (n) at (0,0)  {$(15)$ $M = \left\{m,j,v,t,e,n\right\}$, $S = \left\{(g,f), (e,d), (d,c), (b,a) \right\}$.  Swap $e$ and $d$ in $M$.};			
		\end{tikzpicture}			
		
		\vspace{1cm}
		
		\begin{tikzpicture}[scale=0.9]
					
							\node [std](a) at (0,0)  {};
								\node [below=2mm of a](La){$a$};
							\node [std](b) at (1,0)  {};	
								\node [below=1mm of b](Lb){$b$};
							\node [std](c) at (2,0) {};	
								\node [below=2mm of c](Lc){$c$};
							\node [std](d) at (3,0)  {};	
								\node [below=1mm of d](Ld){$d$};
							\node [std](e) at (4,0)  {};	
								\node [below=2mm of e](Le){$e$};
							\node [std](f) at (5,0)  {};	
								\node [below=1mm of f](Lf){$f$};
							\node [std](g) at (6,0) {};	
								\node [below=2mm of g](Lg){$g$};
							\node [std](h) at (7,0)  {};
								\node [below=1.5mm of h](Lh){$h$};	
							\node [std](i) at (8,0)  {};	
								\node [below=1.5mm of i](Li){$i$};
							\node [std](j) at (9,0) {};						
								\node [below=1.5mm of j](Lj){$j$};
							\node [std](k) at (10,0) {};						
								\node [below=1.5mm of k](Lk){$k$};
							\node [std](l) at (11,0) {};						
								\node [below=1.5mm of l](Ll){$l$};
							\node [std](m) at (12,0) {};						
								\node [below=2mm of m](Lm){$m$};																								
								
						\draw (a)--(m);		
						
							\node [std](n) at (0.5,1)  {};	
								\node [above=1.5mm of n](Ln){$n$};
							\node [std](o) at (1.5,1) {};						
								\node [above=1.5mm of o](Lo){$o$};
							\node [std](p) at (3,1) {};						
								\node [left=0.5mm of p](Lp){$p$};
							\node [std](q) at (4,1) {};						
								\node [left=0.5mm of q](Lq){$q$};
							\node [std](r) at (6,1) {};						
								\node [right=0.5mm of r](Lr){$r$};					
							\node [std](s) at (7,1) {};						
								\node [right=1mm of s](Ls){$s$};						
							
							\draw (b)--(n);
							\draw (b)--(o);
							\draw (d)--(p);
							\draw (e)--(q);
							\draw (g)--(r);
							\draw (g)--(s);				
							
							\node [std](t) at (3.5,2) {};						
								\node [above=1.5mm of t](Lt){$t$};
							\node [std](u) at (4.5,2) {};						
								\node [above=1.5mm of u](Lu){$u$};
							\node [std](v) at (6,2) {};						
								\node [above=1.5mm  of v](Lq){$v$};
							\node [std](w) at (7,2) {};						
								\node [above=1.5mm  of w](Lw){$w$};					
							\node [std](x) at (8,2) {};						
								\node [above=1.5mm  of x](Lx){$x$};	
								
							\draw(q)--(t);
							\draw(u)--(q);
							\draw(r)--(v);
							\draw(s)--(w);
							\draw(s)--(x);	
				
				\node [mp](m1) at (m)  {};		
				\node [mp](m2) at (j)  {};		
				\node [mp](m3) at (v)  {};		
				\node [mp](m4) at (t)  {};		
				\node [mp](m5) at (d)  {};		
				\node [mp](m6) at (n)  {};		
							
				\node[below=1cm of g][color=black](BLa){$(16)$ A maximum multipacking $M$ of  $T'$.};				
				\end{tikzpicture}
				
\caption{Example of Algorithm \ref{algo}.}
\label{fig:algo3}
\end{figure}
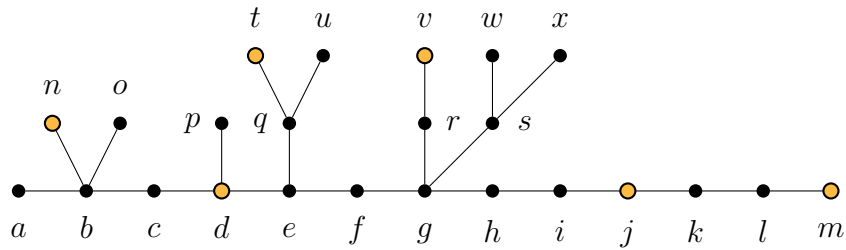



\end{document}